\documentclass[reqno,11pt]{amsart}
\usepackage{amssymb, amsmath,latexsym,amsfonts,amsbsy, amsthm}
\usepackage{color}

\let\pa=\partial
\let\al=\alpha

\let\f=\frac

\let\ve=\varepsilon
\let\pa=\partial

\def\no{\noindent}
\def\non{\nonumber}
\def\na{\nabla}

\def\eqdefa{\buildrel\hbox{\footnotesize def}\over =}

\newcommand{\beq}{\begin{equation}}
\newcommand{\eeq}{\end{equation}}
\newcommand{\ben}{\begin{eqnarray}}
\newcommand{\een}{\end{eqnarray}}
\newcommand{\beno}{\begin{eqnarray*}}
\newcommand{\eeno}{\end{eqnarray*}}

\renewcommand{\theequation}{\thesection.\arabic{equation}}

\newtheorem{theorem}{Theorem}[section]

\newtheorem{lemma}[theorem]{Lemma}
\newtheorem{proposition}[theorem]{Proposition}

\newtheorem{Theorem}{Theorem}[section]

\newtheorem{Lemma}[Theorem]{Lemma}

\newtheorem{Remark}[Theorem]{Remark}

\newcommand{\ud}{\mathrm{d}}

\newcommand{\hh}{\mathbf{h}}

\newcommand{\mm}{\mathbf{m}}
\newcommand{\nn}{\mathbf{n}}

\newcommand{\vv}{\mathbf{v}}

\newcommand{\xx}{\mathbf{x}}

\newcommand{\BB}{\mathbf{B}}
\newcommand{\CC}{\mathbf{C}}
\newcommand{\DD}{\mathbf{D}}
\newcommand{\FF}{\mathbf{F}}
\newcommand{\GG}{\mathbf{G}}
\newcommand{\HH}{\mathbf{H}}
\newcommand{\II}{\mathbf{I}}

\newcommand{\MM}{\mathbf{M}}
\newcommand{\NN}{\mathbf{N}}
\newcommand{\QQ}{\mathbf{Q}}

\newcommand{\CA}{\mathcal{A}}

\newcommand{\CH}{\mathcal{H}}
\newcommand{\CJ}{\mathcal{J}}

\newcommand{\CL}{\mathcal{L}}

\newcommand{\BS}{{\mathbb{S}^2}}
\newcommand{\BR}{{\mathbb{R}^3}}

\newcommand{\BOm}{\mathbf{\Omega}}

\newcommand{\Be}{{\overline{\mathbf{B}}^{\ve}}}
\newcommand{\Qe}{{\overline{\mathbf{Q}}^{\ve}}}

\newcommand{\Ef}{\mathfrak{E}}
\newcommand{\Ff}{\mathfrak{F}}

\newcommand{\tv}{\tilde{\mathbf{v}}}

\newcommand{\wD}{\widetilde{\mathbf{D}}}
\newcommand{\wO}{\widetilde{\mathbf{\Omega}}}

\newcommand{\QI}{{\mathbb{Q}^\mathrm{in}_\nn}}
\newcommand{\QO}{{\mathbb{Q}^\mathrm{out}_\nn}}
\newcommand{\CPi}{{\mathcal{P}^\mathrm{in}}}
\newcommand{\CPo}{{\mathcal{P}^\mathrm{out}}}
\newcommand{\BQ}{{\mathbb{Q}}}

\newcommand{\tr}{\mathrm{Tr}}

\begin{document}
\title[From Landau-de Gennes to Ericksen-Leslie]
{Rigorous derivation from Landau-de Gennes theory to Ericksen-Leslie theory}

\author{Wei Wang}
\address{Beijing International Center for Mathematical Research, Peking University, Beijing 100871, China}
\email{wangw07@pku.edu.cn}

\author{Pingwen Zhang}
\address{School of  Mathematical Sciences and LMAM, Peking University, Beijing 100871, China}
\email{pzhang@pku.edu.cn}

\author{Zhifei Zhang}
\address{School of  Mathematical Sciences and LMAM, Peking University, Beijing 100871, China}
\email{zfzhang@math.pku.edu.cn}

\date{\today}%7.2
\maketitle
\begin{abstract}
Starting from Beris-Edwards system for the liquid crystal, we present a rigorous derivation of Ericksen-Leslie system with general Ericksen stress and Leslie stress by using the Hilbert expansion method.
\end{abstract}

\renewcommand{\theequation}{\thesection.\arabic{equation}}
\setcounter{equation}{0}
%%%%%%%%%%%%%%%%%%%%%%%%%%%%%%%%%%%%%%%%%%%%%%
%%%%%%%%%%%%%%%%%%%%%%%%%%%%%%%%%%%%%%%%%%
\section{Introduction}

Liquid crystals are a state of matter that have properties between
those of a conventional liquid and those of a solid crystal. One of
the most common liquid crystal phases is the nematic. The nematic liquid crystals are composed of rod-like molecules with the long axes of neighboring molecules approximately aligned to one another. There are three different kinds of theories to model the nematic liquid crystals: Doi-Onsager theory, Landau-de Gennes theory and Ericksen-Leslie
theory. The first is the molecular kinetic theory, and the later two are the continuum theory.
In the spirit of Hilbert sixth problem, it is very important to explore the relationship between these theories.

Ball-Majumdar \cite{BM} define a Landau-de Gennes type energy functional
in terms of the mean-field Maier-Saupe energy.  Majumdar-Zarnescu \cite{MZ} consider the Oseen-Frank limit of the static Q-tensor model.
Their results show that the predictions of the Oseen-Frank theory
and the Landau-De Gennes theory agree away from the singularities of the limiting Oseen-Frank global minimizer.

In \cite{KD, EZ}, Kuzzu-Doi and E-Zhang formally derive the Ericksen-Leslie
equation from the Doi-Onsager equations by taking small Deborah number limit.
In our recent work \cite{WZZ1},  we justify their formal derivation before the first singularity time of the Ericksen-Leslie system.
In \cite{HLWZ,WZZ3}, a systematical approach was proposed to derive the continuum theory from the molecular kinetic theory in static and dynamic case.

The goal of this work is to present a rigorous derivation from Landau-de Gennes theory  to Ericksen-Leslie theory. Let us first give a brief introduction to two theories \cite{DG, Doi}.

\subsection{Landau-de Gennes theory}

In this theory,  the state of the nematic liquid crystals is described by the macroscopic Q-tensor order parameter,
which is a symmetric, traceless $3\times 3$ matrix. Physically, it can be interpreted as the second-order moment of
the orientational distribution function $f$, that is,
\beno
\QQ=\int_\BS(\mm\mm-\frac{1}{3}\II)f\ud\mm.
\eeno
When $\QQ=0$, the nematic liquid crystal is said to be isotropic. When $\QQ$ has two equal non-zero eigenvalues, it is said to be uniaxial and
$\QQ$ can be written as
\beno
\QQ=s\big(\nn\nn-\f13\II\big),\quad \nn\in \BS.
\eeno
When $\QQ$ has three distinct eigenvalues, it is said to be biaxial and $\QQ$ can be written as
\beno
\QQ=s\big(\nn\nn-\f13\II\big)+\lambda(\nn'\nn'-\frac13\II),\quad \nn,\,\nn'\in \BS,\quad \nn\cdot\nn'=0.
\eeno

The general Landau-de Gennes energy functional takes the form
\begin{align}
\nonumber\mathcal{F}(\QQ,\nabla\QQ)=&\int_{\BR}\Big\{ \underbrace{-\frac{a}2\tr\QQ^2
-\frac{b}{3}\tr\QQ^3+\frac{c}{4}\mathrm{Tr}\QQ^4}_{F_b:\text{bulk energy}}\\
&+\underbrace{\frac{1}{2}\Big(L_1|\nabla\QQ|^2+L_2Q_{ij,j}Q_{ik,k}
+L_3Q_{ij,k}Q_{ik,j}+L_4Q_{ij}Q_{kl,i}Q_{kl,j}\Big)}_{F_e:\text{elastic energy}} \Big\}\ud\xx.\label{eq:Landau-energy}
\end{align}
Here $a, b, c$ are material-dependent and temperature-dependent nonnegative constants and $L_i(i=1,2,3,4)$ are material dependent elastic constants.
We refer to \cite{DG, MN} for more details.

There are several dynamic Q-tensor models to describe the flow of the nematic liquid crystal,
which are either derived from the molecular kinetic theory for the rigid rods by various closure approximations
such as \cite{Feng, FLS, WZZ3}, or directly derived by variational method such as Beris-Edwards model \cite{BE} and Qian-Sheng's  model \cite{QS}. In this work, we will use  Beris-Edwards model, which takes the form
\begin{align}\label{eq:BE-v}
&\frac{\partial \vv}{\partial t}+\vv\cdot\nabla\vv=-\nabla p+\nabla\cdot(\sigma^{s}+\sigma^{a}+\sigma^d),\quad \\
&\nabla\cdot\vv=0,\\\label{eq:BE-Q}
&\frac{\partial \QQ}{\partial t}+\vv\cdot\nabla \QQ +\QQ\cdot\BOm-\BOm\cdot \QQ=\frac{1}{\Gamma}\HH +S_\QQ(\DD).
\end{align}
Here  $\vv$ is the velocity of the fluid, $p$ is the pressure, $\Gamma$ is a collective rotational
diffusion constant, $\DD=\frac{1}{2}(\na \vv+(\na \vv)^T), \BOm=\frac{1}{2}(\na \vv-(\na \vv)^T)$; $\sigma^{s}$, $\sigma^{a}$ and $\sigma^d$ are symmetric viscous stress,
anti-symmetric viscous stress and distortion stress respectively defined by
\begin{align}
\sigma^{s}=\eta\DD-S_{\QQ}(\HH),\quad \sigma^{a}=\QQ\cdot\HH-\HH\cdot\QQ,\quad
\sigma^d_{ij}=-\frac{\partial \mathcal{F}}{\partial Q_{kl,j}}Q_{kl,i},\nonumber
\end{align}
where $\eta>0$ is the viscous coefficient,  $\HH$ is the molecular field given by
\begin{align}
\HH(\QQ)=-\frac{\delta \mathcal{F}}{\delta \QQ}
=-\frac{\partial {F}_b}{\partial\QQ}+\partial_i\Big(\frac{\partial {F}_e}{\partial\QQ_{,i}}\Big),\nonumber
\end{align}
and $S_\QQ(\MM)$ is defined by
\begin{align}
S_\QQ(\MM)=\xi\Big(\MM\cdot(\QQ+\frac13\II)+(\QQ+\frac13\II)\cdot\MM-2(\QQ+\frac13\II)\QQ:\MM\Big)\nonumber
\end{align}
for symmetric and traceless matrix $\MM$, where $\xi$ is a constant depending on the molecular details of a given liquid crystal.

We refer to \cite{PZ1,PZ2} for the well-posedness results of the Q-tensor model.

\subsection{Ericksen-Leslie theory}

The hydrodynamic theory of liquid crystals was established by  Ericksen and Leslie in the 1960's \cite{E-61, Les}.
In this theory, the configuration of the liquid crystals is described by a director
field $\nn\in \BS$. The general Ericksen-Leslie system takes the form
\begin{align}
&\vv_t+\vv\cdot\nabla\vv=-\nabla{p}+\nabla\cdot\sigma,\label{eq:EL-v}\\
&\na\cdot\vv=0,\\
&\nn\times\big(\hh-\gamma_1\NN-\gamma_2\DD\cdot\nn\big)=0.\label{eq:EL-n}
\end{align}
Here the stress $\sigma$ is modeled by the phenomenological constitutive relation
\beno
\sigma=\sigma^L+\sigma^E,
\eeno
where $\sigma^L$ is the viscous (Leslie) stress
\begin{eqnarray}\label{eq:Leslie stress}
\sigma^L=\alpha_1(\nn\nn:\DD)\nn\nn+\alpha_2\nn\NN+\alpha_3\NN\nn+\alpha_4\DD
+\alpha_5\nn\nn\cdot\DD+\alpha_6\DD\cdot\nn\nn \end{eqnarray}
with
\beno
\NN=\nn_t+\vv\cdot\nabla\nn-\BOm\cdot\nn.
\eeno
The six constants $\al_1, \cdots, \al_6$ are called the Leslie coefficients.  While, $\sigma^E$ is the elastic (Ericksen) stress
\begin{eqnarray}\label{eq:Ericksen}
\sigma_{ij}^E=-\frac{\partial{E_F}}{\partial n_{k,j}}n_{k,i},
\end{eqnarray}
where $E_F=E_F(\nn,\nabla\nn)$ is the Oseen-Frank energy with the form
\begin{align}\label{energy-OF}
E_F=\f {k_1} 2(\na\cdot\nn)^2+\f {k_2} 2(\nn{\cdot}(\na\times\nn))^2
+\f {k_3} 2|\nn{\times}(\na\times \nn)|^2
+\frac{k_2+k_4}2\big(\textrm{tr}(\na\nn)^2-(\na\cdot\nn)^2\big).
\end{align}
Here $k_1, k_2, k_3, k_4$ are the elastic constant.
The molecular field $\hh$ is given by
\beno
&&\hh=-\frac{\delta{E_F}}{\delta{\nn}}=
\nabla\cdot\frac{\partial{E_F}}{\partial(\nabla\nn)}-\frac{\partial{E_F}}{\partial\nn}.
\eeno
Finally, the Leslie coefficients and $\gamma_1, \gamma_2$ satisfy the following relations
\begin{eqnarray}
&\alpha_2+\alpha_3=\alpha_6-\alpha_5,\label{Leslie relation}\\
&\gamma_1=\alpha_3-\alpha_2,\quad \gamma_2=\alpha_6-\alpha_5,\label{Leslie-coeff}
\end{eqnarray}
where (\ref{Leslie relation}) is called Parodi's relation derived from the Onsager reciprocal relation \cite{Parodi}. These two relations will
ensure that the system (\ref{eq:EL-v})--(\ref{eq:EL-n}) has a basic energy law:
\begin{align}
-\frac{\ud}{\ud{t}}\Big(\int_{\BR}\frac{1}{2}|\vv|^2\ud\xx+E_F\Big)
=&\int_{\BR}\Big((\alpha_1+\frac{\gamma_2^2}{\gamma_1})(\DD:\nn\nn)^2
+\alpha_4|\DD|^2\qquad\nonumber\\
&\quad+\big(\alpha_5+\alpha_6-\frac{\gamma_2^2}{\gamma_1}\big)|\DD\cdot\nn|^2
+\frac{1}{\gamma_1}|\nn\times\hh|^2\Big)\ud\xx.\quad\label{EL_energy_law}
\end{align}

We refer to \cite{LL,WZZ2} for the well-posedness results of the Ericksen-Leslie system. In \cite{WZZ2}, we proved the well-posedness of the system under a natural
physical condition on the Leslie coefficients, and in \cite{HLW,WW},  the authors proved the global existence of weak solution in 2-D case.

\subsection{Main result: from Beris-Edwards system to Ericksen-Leslie system}

Since the elastic constants $L_i(i=1,2,3,4)$ are typically very small compared with $a, b,c$, we introduce a small parameter $\varepsilon$ and consider the following
Landau-de Gennes energy functional
\begin{align}
\nonumber\mathcal{F}_\ve(\QQ,\nabla\QQ)=&\frac1\ve\int_\BR\Big\{ \underbrace{-\frac{a}2\tr\QQ^2
-\frac{ b}{3}\tr\QQ^3+\frac{c}{4}\mathrm{Tr}\QQ^4}_{F_b(\QQ)}\Big\}\ud\xx
\\&
+\int_\BR\underbrace{\frac{1}{2}\Big(L_1|\nabla\QQ|^2+L_2Q_{ij,j}Q_{ik,k}
+L_3Q_{ij,k}Q_{ik,j}+L_4Q_{ij}Q_{kl,i}Q_{kl,j}\Big)}_{F_e(\QQ)} \ud\xx.
\end{align}
In the case when $L_4\neq0$, the term $Q_{ij}Q_{kl,i}Q_{kl,j}$ may cause the energy to be not bounded from below \cite{BM}.
Therefore, we only consider the case $L_4=0$. Furthermore, we assume
\begin{align}\label{ass:L}
L_1>0,\quad L_1+L_2+L_3>0.
\end{align}
which will ensure that the elastic energy is strictly positive (see Lemma \ref{lem:L}).

We introduce two operators
\begin{align*}
&\mathcal{J}(\QQ)\eqdefa \frac{\delta F_b(\QQ)}{\delta \QQ}=-a\QQ-b\QQ^2+c|\QQ|^2\QQ+\frac13b|\QQ|^2\II,\\
&(\mathcal{L}(\QQ))_{kl}\eqdefa-\partial_i\big(\frac{\partial {F}_e}{\partial Q_{kl,i}}\big) = -\big(L_1\Delta Q_{kl}
+\frac12(L_2+L_3)(Q_{km,ml}+Q_{lm,mk}-\frac23\delta_{kl}Q_{ij,ij})\big),
\end{align*}
and define the tensor $\sigma^d(\QQ,\widetilde\QQ)$ as
\beno
\sigma^d_{ji}(\QQ,\widetilde\QQ)\eqdefa-\frac{\partial \mathcal{F}_\ve}{\partial Q_{kl,j}}\widetilde Q_{kl,i}
=-\big(L_1Q_{kl,j}\widetilde Q_{kl,i}+L_2Q_{km,m}\widetilde Q_{kj,i}+L_3Q_{kj,l}\widetilde Q_{kl,i}\big).
\eeno
So, the molecular field and distortion stress can be written as
\beno
\HH_\ve(\QQ)=-\frac1\ve\mathcal{J}(\QQ)-\mathcal{L}(\QQ),\qquad \sigma^d=\sigma^d(\QQ,\QQ).
\eeno

We study the Beris-Edwards system with a small parameter $\ve$:
\begin{align}\label{eq:BE-ve}
&\frac{\partial \vv^\ve}{\partial t}+\vv^\ve\cdot\nabla\vv^\ve=-\nabla p^\ve+\nabla\cdot(\sigma^{s}_\ve+\sigma^{a}_\ve+\sigma^d_\ve),\quad \\
&\nabla\cdot\vv^\ve=0,\\\label{eq:BE-Qe}
&\frac{\partial \QQ^\ve}{\partial t}+\vv^\ve\cdot\nabla \QQ^\ve +\QQ^\ve\cdot\BOm^\ve-\BOm^\ve\cdot \QQ^\ve=\frac{1}{\Gamma}\HH_\ve +S_{\QQ^\ve}(\DD^\ve),
\end{align}
where $\DD^\ve=\frac{1}{2}(\na \vv^\ve+(\na \vv^\ve)^T),\, \BOm^\ve=\frac{1}{2}(\na \vv^\ve-(\na \vv^\ve)^T)$, and
\begin{align}
\sigma^{s}_\ve=\eta\DD^\ve-S_{\QQ^\ve}(\HH_\ve),\quad \sigma^{a}_\ve=\QQ^\ve\cdot\HH_\ve-\HH_\ve\cdot\QQ^\ve,\quad
\sigma^d_\ve=\sigma^d(\QQ^\ve,\QQ^\ve).\nonumber
\end{align}

Our main result is stated as follows.

\begin{theorem}\label{thm:main}

Let $(\nn(t,\xx), \vv(t,\xx))$ be a solution of the Ericksen-Leslie system (\ref{eq:EL-v})--(\ref{eq:EL-n}) on $[0,T]$
with the coefficients $k_i(i=1,2,3,4)$ and $\al_i(i=1,\cdots,6)$
given by (\ref{OF-LD-relation})-(\ref{Leslie-coef}), which satisfies
\beno
\vv\in C([0,T];H^{k}), \quad \nabla\nn\in C([0,T];H^{k})\quad \textrm{for}\quad k\ge 20.
\eeno
Let $\QQ_0(t,x)=s\big(\nn(t,\xx)\nn(t,\xx)-\II\big)$ with $s=\frac {b+\sqrt{b^2+24ac}} {4c}$, and
the functions $\big(\QQ_1,\QQ_2,\QQ_3, \vv_1,\vv_2\big)$ are determined by Proposition \ref{prop:Hilbert}.
Assume that the initial data $(\QQ^{\ve}_0, \vv^\ve_0)$ takes the form
\begin{align*}
\QQ_0^\ve(\xx)=&\QQ_{0}(0,\xx)+\ve\QQ_{1}(0,\xx)+\ve^2\QQ_{2}(0,\xx)+\ve^3\QQ_{3}(0,\xx)+\ve^3\QQ_{0R}^\ve(\xx),\\
\vv_0^\ve(\xx)=&\vv_0(0,\xx)+\ve\vv_{1}(0,\xx)+\ve^2\vv_{2}(0,\xx)+\ve^3\vv_{0R}^\ve(\xx),
\end{align*}
where  $(\QQ_{0R}^\ve, \vv_{0R}^\ve)$ satisfies
\begin{align}
\|\vv_{0R}^\ve\|_{H^2}+\|\QQ_{0R}^\ve\|_{H^3}+\ve^{-1}\|\CPo(\QQ^\ve_{0R})\|_{L^2}\le E_0.\non
\end{align}
Then there exists $\ve_0>0$ and $E_1>0$ such that
for all $\ve<\ve_0$, the system (\ref{eq:BE-ve})--(\ref{eq:BE-Qe}) has a unique solution
$(\QQ^\ve(t,\xx), \vv^\ve(t,\xx))$ on $[0,T]$ which has the expansion
\begin{align*}
\QQ^\ve(t,\xx)=&\QQ_0(t,\xx)+\ve\QQ_1(t,\xx)+\ve^2\QQ_2(t,\xx)+\ve^3\QQ_3(t,\xx)+\ve^3\QQ_R^\ve(t,\xx),\\
\vv^\ve(t,\xx)=&\vv_0(t,\xx)+\ve\vv_1(t,\xx)+\ve^2\vv_2(t,\xx)+\ve^3\vv_R^\ve(t,\xx),
\end{align*}
where $(\QQ_R^\ve,\vv_R^\ve)$ satisfies
\begin{align*}
\Ef(\QQ_R^\ve(t),\vv_R^\ve(t))\le E_1.
\end{align*}
Here $\Ef$ is defined by
\begin{align}\nonumber
\mathfrak{E}(\QQ,\vv)\eqdefa&~\int\Big(|\vv|^2+\frac1\ve\CH_\nn^\ve(\QQ):\QQ+|\QQ|^2\Big)
+\ve^2\Big(|\nabla\vv|^2+\frac1\ve\CH_\nn^\ve(\nabla\QQ):\nabla\QQ\Big)\\\label{eq:energy functional}
&\quad+\ve^4\Big(|\Delta\vv|^2+\frac1\ve\CH_\nn^\ve(\Delta\QQ):\Delta\QQ\Big)\ud\xx,
\end{align}
and $\CH^\ve_\nn(\QQ)=\CH_{\nn}(\QQ)+\ve\CL(\QQ)$, where $\CH_{\nn}$ is the linearized operator of $\CJ$ around $\QQ_0$.

\end{theorem}

\begin{Remark}
It is known \cite{BE, WZZ2} that the energy (\ref{EL_energy_law}) is dissipated or equivalently
\begin{align}\label{eq:dissipation}
\beta_1(\nn\nn:\DD)^2+\beta_2\DD:\DD+\beta_3|\DD\cdot\nn|^2>0
\end{align}
for any non-zero symmetric traceless matrix $\DD$ and unit vector $\nn$, if and only if
\begin{align}\label{Leslie condition}
\beta_2>0,\quad2\beta_2+\beta_3>0,\quad\frac{3}{2}\beta_2+\beta_3+\beta_1>0,
\end{align}
where
\begin{align}
\beta_1=\alpha_1+\frac{\gamma_2^2}{\gamma_1},\quad\beta_2=\alpha_4,
\quad\beta_3=\alpha_5+\alpha_6-\frac{\gamma_2^2}{\gamma_1}.\non
\end{align}
In \cite{WZZ2}, we proved the well-posedness of the system (\ref{eq:EL-v})--(\ref{eq:EL-n}) under the condition (\ref{Leslie condition})
and in the case when $k_1=k_2=k_3$ and $k_4=0$. Wang-Wang \cite{WW} generalize our result to the case with general Oseen-Frank energy under the condition
\beno
\min(k_1,k_2,k_3)>0.
\eeno
By Remark \ref{rem:dissipation},  the energy for the Ericksen-Leslie system derived from the Beris-Edwards system is dissipated, and
by (\ref{OF-LD-relation}) and (\ref{ass:L}), $\min(k_1,k_2,k_3)>0$.
Thus, it is well-posed.
\end{Remark}

\begin{Remark}
The same result should be true for Qian-Sheng's model in \cite{QS}.
\end{Remark}

Let us conclude this section by presenting a sketch of the proof. \vspace{0.1cm}

The first step is to make a formal expansion for the solution $(\vv^\ve,\QQ^\ve)$:
\begin{align*}
&\QQ^\ve(t,\xx)=\QQ_0(t,\xx)+\ve\QQ_1(t,\xx)+\ve^2\QQ_2(t,\xx)+\ve^3\QQ_3(t,\xx)+\ve^3\QQ_R(t,\xx),\\
&\vv^\ve(t,\xx)=\vv_0(t,\xx)+\ve\vv_1(t,\xx)+\ve^2\vv_2(t,\xx)+\ve^3\vv_R(t,\xx).
\end{align*}
We find that $\mathcal{J}(\QQ_0)=0,$ and Proposition \ref{prop:critical ponit} ensures
$$
\QQ_0=s(\nn\nn-\frac13\II),
$$
for some $\nn\in \BS$ and $s=\frac {b\pm\sqrt{b^2+24ac}} {4c}$. By studying the kernel of the linearized operator $\CH_\nn$,
it can be proved that $(\vv_0,\nn)$ is a solution of the Ericksen-Leslie system. The existence of $(\QQ_i,\vv_i)$ for $i\neq 0$ is also nontrivial,
since they  satisfy a system with the complicated dissipation relation.

The most difficult step is to show that the remainder $(\vv_R,\QQ_R)$ is uniformly bounded in $\ve$,
which satisfies (dropping good error terms)
\begin{align}
\frac{\partial \vv_R}{\partial t}=&-\nabla p_R+\eta\Delta\vv_R
+\nabla\cdot\Big(\frac1\ve S_{\QQ_0}(\HH_R)
-\frac1\ve\QQ_0\cdot \HH_R+\frac1\ve\HH_R\cdot\QQ_0\Big),\non\\
\frac{\partial \QQ_R}{\partial t} =
&-\frac{1}{\Gamma\ve}\CH_{\nn}^\ve(\QQ_R)
+S_{\QQ_0}\DD_R+\BOm_R\cdot\QQ_0-\QQ_0\cdot\BOm_R.\non
\end{align}
This is a system with the singular terms of order $\f 1 \ve$.  To deal with them, we introduce a key energy functional $\Ef$ defined by (\ref{eq:energy functional}). Then we prove that $\Ef$ is uniformly bounded by the energy method,
where main difficulty is to control the terms like
\beno
\frac{1}{\ve}\big\langle\partial_t(\nn\nn)\cdot\QQ_R,
\QQ_R\big\rangle.
\eeno
A rough estimate gives
\beno
\frac{1}{\ve}\big\langle\partial_t(\nn\nn)\cdot\QQ_R,
\QQ_R\big\rangle\le C\ve^{-1}\|\QQ_R\|_{L^2}^2\le C\ve^{-1}\Ef,
\eeno
which is obviously unacceptable. Surprisingly, it can be proved that for any $\delta>0$
\beno
\frac{1}{\ve}\big\langle\partial_t(\nn\nn)\cdot\QQ_R,
\QQ_R\big\rangle\le C_\delta\Ef+\delta\Ff,
\eeno
where $\Ff$ is the dissipation part in the energy estimates. The proof relies on the fact that
the linearized operator $\CH_{\nn}$ is an 1-1 map outside its kernel, and its inverse $\CH_{s,\nn}^{-1}$ can be explicitly given
(see Proposition \ref{prop:H-inverse}).

\vspace{0.3cm}

\no{\bf Notations.}\,\,For any two vectors $\mm=(m_1,m_2,m_3),\nn=(n_1,n_2,n_3)\in\BR$, we denote the tensor product by
$
\mm\otimes\nn=[m_in_j]_{1\le i,j\le 3}.
$
In the sequel, we use $\mm\nn$ to denote $\mm\otimes\nn$ for simplicity when no ambiguity is possible.
$A\cdot B$ denotes the usual matrix/vector-matrix/vector product. $A:B$ denotes $\tr(AB)=A_{ij}B_{ji}$.
The divergence of a tensor is defined by $\nabla\cdot\sigma=\partial_j\sigma_{ij}$. We also use
$f_{,i}$ to denote $\partial_if$ for simplicity.

\setcounter{equation}{0}
\section{Critical points and the linearized operator}

\subsection{Critical points of $F_b(\QQ)$} We say that a matrix $\QQ_0$ is a critical point of $F_b(\QQ)$ if $\CJ(\QQ_0)=0$.
We have the following characterization for critical points(see also \cite{BM} and references therein).

\begin{proposition}\label{prop:critical ponit}
$\CJ(\QQ)=0$ if and only if
\begin{align*}
\QQ=s(\nn\nn-\frac13\II),
\end{align*}
for some $\nn\in \BS$ and $s=0$ or is a solution of $2cs^2-bs-3a=0$, that is,
\beno
s_{1,2}=\frac {b\pm\sqrt{b^2+24ac}} {4c}.
\eeno
\end{proposition}

\begin{proof}
Since $\QQ$ is symmetric and traceless, we may write
\begin{align*}
\QQ=\lambda_1\nn_1\otimes\nn_1+\lambda_2\nn_2\otimes\nn_2+\lambda_3\nn_3\otimes\nn_3,
\end{align*}
where $\lambda_1,\lambda_2,\lambda_3$ are eigenvalues with $\lambda_1+\lambda_2+\lambda_3=0$,
and $\nn_1,\nn_2,\nn_3$ are the corresponding eigenvectors satisfying $\nn_i\cdot\nn_j=\delta_{ij}.$
A direct computation gives
\begin{align*}
\CJ(\QQ)=\frac b3(\lambda_1^2+\lambda_2^2+\lambda_3^2)\II+\sum_{i=1}^3
\Big(-a\lambda_i-b\lambda_i^2+c(\lambda_1^2+\lambda_2^2+\lambda_3^2)\lambda_i\Big)\nn_i\otimes\nn_i.
\end{align*}
So, $\CJ(\QQ)=0$ if and only if $\CJ(\QQ)\cdot\nn_i=0$ for $i=1,2,3$, which is equivalent to
\begin{align*}
\rho_i\triangleq\frac b3(\lambda_1^2+\lambda_2^2+\lambda_3^2)-a\lambda_i-b\lambda_i^2+c(\lambda_1^2+\lambda_2^2+\lambda_3^2)\lambda_i=0
\quad \text{for }i=1,2,3.
\end{align*}

If $\lambda_i$ are all equal, then $\lambda_1=\lambda_2=\lambda_3=0$, hence $\QQ=0$.
If not, we may assume $\lambda_1\neq\lambda_2$ without loss of generality. Due to $\lambda_1+\lambda_2+\lambda_3=0$, we get
\begin{align*}
\rho_1=\frac b3\Big(\lambda_1^2+\lambda_2^2+(\lambda_1+\lambda_2)^2\Big)-a\lambda_1
-b\lambda_1^2+c\Big(\lambda_1^2+\lambda_2^2+(\lambda_1+\lambda_2)^2\Big)\lambda_1=0,\\
\rho_2=\frac b3\Big(\lambda_1^2+\lambda_2^2+(\lambda_1+\lambda_2)^2\Big)-a\lambda_2
-b\lambda_2^2+c\Big(\lambda_1^2+\lambda_2^2+(\lambda_1+\lambda_2)^2\Big)\lambda_2=0.
\end{align*}

Let $r=\lambda_1+\lambda_2, R=\lambda_1^2+\lambda_2^2+(\lambda_1+\lambda_2)^2$. From the fact
$\rho_1-\rho_2=0$, we infer
\begin{align}
cR=a+br,\label{equation for r}
\end{align}
and from $\rho_1+\rho_2=0$, we infer
\begin{align*}
\frac{2}3bR-ar-b(R-r^2)+cRr=0,
\end{align*}
which imply $\frac b3R=2br^2$ or $b(2\lambda_1+\lambda_2)(\lambda_1+2\lambda_2)=0$.
Without loss of generality, we assume $\lambda_1=-2\lambda_2$, then $\lambda_3=\lambda_2$.
Then using the identity $\nn_1\otimes\nn_1+\nn_2\otimes\nn_2+\nn_3\otimes\nn_3=\II$, we get
\beno
\QQ=s(\nn\nn-\frac13\II),
\eeno
where $s=-3\lambda_2$. We know from (\ref{equation for r}) that $s$ satisfies $2cs^2-bs-3a=0$.
\end{proof}

\subsection{The linearized operator of $\CJ$}

Given a critical point $\QQ_0=s(\nn\nn-\frac13\II)$, the linearized operator $\CH_{\QQ_0}$
of $\CJ$ around $\QQ_0$ is given by
\begin{align}
\CH_{\QQ_0}(\QQ)&=a\QQ-b\big(\QQ_0\cdot\QQ+\QQ\cdot\QQ_0-\frac23(\QQ_0:\QQ)\II\big)+c\big(|\QQ_0|^2\QQ+2(\QQ_0:\QQ)\QQ_0\big).\non
\end{align}
Putting $\QQ_0=s(\nn\nn-\frac13\II)$ into the above formula and using the equation $2cs^2-bs-3a=0$, we find that
\ben\label{def:Hn}
\CH_{\QQ_0}(\QQ)=bs\big(\QQ-(\nn\nn\cdot\QQ+\QQ\cdot\nn\nn)+\frac{2}{3}(\QQ:\nn\nn)\II\big)
+2cs^2(\QQ:\nn\nn)(\nn\nn-\frac13\II).
\een
In the sequel, we denote $\CH_{\QQ_0}$ by $\CH_{s, \nn}$ for the simplicity.

We denote by $\mathbb{Q}$  the Hilbert space of symmetric traceless matrix with the following inner product:
\begin{align}
\langle \QQ^1,\QQ^2\rangle\eqdefa \tr(\QQ^1\QQ^2),\non
\end{align}
which is a five-dimensional space.
For a given $\nn\in\BS$, we define a two-dimensional space $\QI$ as
\begin{align}
\QI&\eqdefa\big\{ \nn\otimes\nn^\bot + \nn^\bot\otimes\nn \in\mathbb{Q}: \nn^\bot\in\mathbb{V}_\nn\big\},
\end{align}
where $\mathbb{V}_\nn\eqdefa\big\{\nn^\bot\in \BR: \nn^\bot\cdot\nn=0\big\}$.
Let $\QO$ be the orthogonal
complement of $\QI$ in $\mathbb{Q}$.
We denote by  $\CPi$ the projection operator from $\BQ$ to $\QI$
and by $\CPo$ the projection operator from $\BQ$ to $\QO$.
Note that
\beno
|\QQ-(\nn\nn^\bot + \nn^\bot\nn)|^2=|\nn^\bot-(\II-\nn\nn)\QQ\cdot\nn|^2+|\QQ|^2-2|\QQ\cdot\nn|^2+2(\QQ:\nn\nn)^2,
\eeno
which means that the left hand side attains minimum when $\nn^\bot=(\II-\nn\nn)\QQ\cdot\nn$. Hence,
\begin{align}
\CPi(\QQ)=&\nn\big[(\II-\nn\nn)\cdot\QQ\cdot\nn\big]+\big[(\II-\nn\nn)\cdot\QQ\cdot\nn\big]\nn\non\\
=&(\nn\nn\cdot\QQ+\QQ\cdot\nn\nn)-2(\QQ:\nn\nn)\nn\nn.\label{def:Pin}
\end{align}
Moreover,
\ben\label{eq:Pin-norm}
|\CPi(\QQ)|^2=2|\QQ\cdot\nn|^2-2(\QQ:\nn\nn)^2.
\een

\begin{proposition}\label{prop:H-lower}
Let $s=\frac {b+\sqrt{b^2+24ac}} {4c}$. Then for any $\nn\in \BS$, it holds that
\begin{itemize}

\item $\CH_{s,\nn}:\mathbb{Q}\rightarrow \QO$, hence  $\CH_{s,\nn}\QI=0$;

\item There exists $c_0=c_0(a, b, c)>0$ such that for any $\QQ\in\QO$,
\beno
\CH_{s,\nn}\QQ :\QQ \ge c_0|\QQ|^2.
\eeno
\end{itemize}
\end{proposition}

\begin{proof}
It is easy to see that
\ben\label{fact:out}
\textbf{A}\triangleq\QQ-(\nn\nn\cdot\QQ+\QQ\cdot\nn\nn)+\frac{2}{3}(\QQ:\nn\nn)\II,\quad \textbf{B}\triangleq(\QQ:\nn\nn)(\nn\nn-\frac13\II)\in \QO.
\een
This gives the first point.

Take $c_0=\min\{bs, 2cs^2-bs\}>0$. Then for any $\QQ\in \mathbb{Q}$, we have
\begin{align*}
\CH_{s,\nn}(\QQ):\QQ=&bs(|\QQ|^2-2|\QQ\cdot\nn|^2+(\QQ:\nn\nn)^2)+(2cs^2-bs)(\QQ:\nn\nn)^2\\
\ge & c_0(|\QQ|^2-2|\QQ\cdot\nn|^2+2(\QQ:\nn\nn)^2).
\end{align*}
We infer from  (\ref{eq:Pin-norm}) that for $\QQ\in \QO$, we have $2|\QQ\cdot\nn|^2-2(\QQ:\nn\nn)^2=0$, hence,
\beno
\CH_{s,\nn}(\QQ):\QQ\ge c_0|\QQ|^2.
\eeno
This proves the second point.
\end{proof}

\begin{proposition}\label{prop:H-inverse}
Let $s=\frac {b+\sqrt{b^2+24ac}} {4c}$. Then
$\CH_{s,\nn}$ is an 1-1 map on $\QO$ and its inverse $\CH_{s,\nn}^{-1}$ is given by
\begin{align}
\CH_{s, \nn}^{-1}(\QQ)=\frac{1}{bs}\Big(\QQ-(\nn\nn\cdot\QQ+\QQ\cdot\nn\nn)+\frac{2}{3}(\QQ:\nn\nn)\II\Big)
+\frac{4b+2cs}{bs(4cs-b)}(\QQ:\nn\nn)(\nn\nn-\frac13\II).\non
\end{align}
\end{proposition}

\begin{proof}
A direct computation gives
\begin{align*}
\CH_{s,\nn}(\QQ):\nn\nn=-\frac13(bs-4cs^2)(\QQ:\nn\nn).
\end{align*}
From this fact, we deduce
\begin{align*}
\CH_{s,\nn}^{-1}\CH_{s,\nn}(\QQ)=&\frac{1}{bs}\Big(\CH_{s,\nn}\QQ-(\nn\nn\cdot\CH_{s,\nn}\QQ+\CH_{s,\nn}\QQ\cdot\nn\nn)+\frac{2}{3}(\CH_{s,\nn}\QQ:\nn\nn)\II\Big)\\
&+\frac{4b+2cs}{bs(4cs-b)}(\CH_{s,\nn}\QQ:\nn\nn)(\nn\nn-\frac13\II)\\
=&\Big(\QQ-(\nn\nn\cdot\QQ+\QQ\cdot\nn\nn)+\frac{2}{3}(\QQ:\nn\nn)\II\Big)
+\frac{2cs^2}{bs}(\QQ:\nn\nn)(\nn\nn-\frac13\II)\\
&+\frac{2}{bs}\Big(\frac13bs\nn\nn(\QQ:\nn\nn)-\frac43cs^2(\QQ:\nn\nn)\nn\nn\Big)-\frac29\frac1{bs}(bs-4cs^2)(\QQ:\nn\nn)\II\\
&-\frac13\frac{4b+2cs}{bs(4cs-b)}(bs-4cs^2)(\QQ:\nn\nn)(\nn\nn-\frac13\II)\\
=&\Big(\QQ-(\nn\nn\cdot\QQ+\QQ\cdot\nn\nn)+2(\QQ:\nn\nn)\nn\nn\Big)\\
&+\Big(-\frac{2cs^2}{3bs}-\frac{4}3-\frac13\frac{4b+2cs}{bs(4cs-b)}(bs-4cs^2)\Big)(\QQ:\nn\nn)(\nn\nn-\frac13\II)\\
=&\QQ-(\nn\nn\cdot\QQ+\QQ\cdot\nn\nn)+2(\QQ:\nn\nn)\nn\nn.
\end{align*}
Then it follows from (\ref{def:Pin}) that for $\QQ\in \QO$,
\beno
\CH_{s,\nn}^{-1}\CH_{s,\nn}(\QQ)=\QQ.
\eeno
That means that $\CH_{s,\nn}$ is an 1-1 map on $\QO$.
\end{proof}

\begin{Remark}
The construction of $\CH_{s,\nn}^{-1}$ is motivated by the fact (\ref{fact:out}).
So, we hope that $\CH_{s,\nn}^{-1}$ has the form
\beno
\CH_{s,\nn}^{-1}(\QQ)=\al\mathbf{A}+\beta\mathbf{B}.
\eeno
Fortunately, we can choose suitable $\al, \beta$ such that
\beno
\CH_{s,\nn}^{-1}\CH_{s,\nn}(\QQ)=\QQ-(\nn\nn\cdot\QQ+\QQ\cdot\nn\nn)+2(\QQ:\nn\nn)\nn\nn.
\eeno
\end{Remark}

\begin{Lemma}\label{lem:L}
Under the assumption (\ref{ass:L}), there exists a positive constant $L_0$ depending only on $L_1,L_2, L_3$ such that
\begin{align}
\int_{\BR}\CL(\QQ):\QQ\ud\xx\ge L_0\|\nabla\QQ\|_{L^2}^2.\non
\end{align}
\end{Lemma}

\begin{proof} Let $\QQ_i=(Q_{i1}, Q_{i2}, Q_{i3})$. By integration by parts, we get
\begin{align*}
\int_{\BR}\CL(\QQ):\QQ\ud\xx=&\int_{\BR}L_1|\nabla\QQ|^2+L_2Q_{ki,i}Q_{kj,j}+L_3Q_{ki,j}Q_{kj,i} \ud\xx\\
=&\sum_{i=1}^3\int_{\BR}L_1|\nabla\QQ_i|^2+(L_2+L_3)|\na\cdot \QQ_i|^2\ud\xx \\
=&\sum_{i=1}^3\int_{\BR}L_1|\nabla\times \QQ_i|^2+(L_1+L_2+L_3)|\na\cdot \QQ_i|^2\ud\xx \\
\ge& \min(L_1, L_1+L_2+L_3)\int_{\BR}|\nabla\QQ|^2\ud\xx.
\end{align*}
This gives the lemma by taking $L_0=\min(L_1, L_1+L_2+L_3)$.
\end{proof}

\begin{Lemma}\label{lem:ker-vanish}
If $\QQ_i\in\QI(i=1,2,3)$, then it holds that
\beno
\tr(\QQ_1\cdot\QQ_2\cdot\QQ_3)=0.
\eeno
Especially, if $\QQ_1,\QQ_2\in\QI$, then $\QQ_1\cdot\QQ_2\in\QO$.
\end{Lemma}
\begin{proof} Assume that $\QQ_i=\nn\nn_i+\nn_i\nn$, where $\nn_i\cdot\nn=0(i=1,2,3)$. Then we have
\begin{align*}
\tr(\QQ_1\cdot\QQ_2\cdot\QQ_3)=\tr\Big(\big(\nn_1\nn_2+\nn\nn(\nn_1\cdot\nn_2)\big)
\cdot(\nn\nn_3+\nn_3\nn)\Big)\\
=(\nn_2\cdot\nn_3)\tr(\nn_1\nn)+(\nn_1\cdot\nn_2)\tr(\nn\nn_3)=0.
\end{align*}
The second statement is obvious.
\end{proof}

\setcounter{equation}{0}

\section{Hilbert expansion}

\subsection{Hilbert expansion}
Let $(\vv^\ve, \QQ^\ve)$ be the solution of (\ref{eq:BE-ve})--(\ref{eq:BE-Qe}). We perform the following so-called Hilbert expansion
\begin{align}
\vv^\ve&=\vv_0+\ve\vv_1+\ve^2\vv_2+\ve^3\vv_R^\ve,\label{eq:v-exp}\\
\QQ^\ve&=\QQ_0+\ve\QQ_1+\ve^2\QQ_2+\ve^3\QQ_3+\ve^3\QQ_R^\ve,\label{eq:Q-exp}
\end{align}
where $\QQ_i\in \BQ(i=0, 1, 2, 3)$ will be determined in what follows.

Let us first make some preliminaries. For $\QQ_i\in \mathbb{M}^{3\times3}(i=1,2,3)$, we denote
\begin{align}
&\BB(\QQ_1,\QQ_2)=\QQ_1\cdot\QQ_2+\QQ_2^T\cdot\QQ_1^T-\frac23\II(\QQ_1:\QQ_2),\non\\
&\CC(\QQ_1,\QQ_2,\QQ_3)=\QQ_1(\QQ_2:\QQ_3)+\QQ_2(\QQ_1:\QQ_3)+\QQ_3(\QQ_1:\QQ_2).\non
\end{align}
It is easy to see that

\begin{Lemma}\label{lem:JH}
For any $\QQ,\tilde\QQ\in \BQ$, it holds that
\begin{align*}
\CJ(\QQ)=a\QQ-\frac{b}2\BB(\QQ,\QQ)+\frac{c}3\CC(\QQ,\QQ,\QQ),\\
\CH_\QQ(\tilde\QQ)=a\tilde\QQ-b\BB(\QQ,\tilde\QQ)+c\CC(\QQ,\QQ,\tilde\QQ).
\end{align*}
\end{Lemma}

It follows from Lemma \ref{lem:JH} that
\begin{align*}
\CJ(\QQ^\ve)=& \CJ(\QQ_0)+\ve\CH_{\QQ_0}(\QQ_1)+a(\ve^2\QQ_2+\ve^3\QQ_3+\ve^3\QQ_R)\\
&-\f b2\sum_{m=0}^{6}\ve^m\sum_{\substack{i+j=m,\,1\le i,j\le 3}}\BB(\QQ_i,\QQ_j)\\
&-b\ve^3\BB(\QQ_0+\ve\QQ_1+\ve^2\QQ_2+\ve^3\QQ_3,\QQ_R)-\frac{b}2\ve^6\BB(\QQ_R,\QQ_R)\\
&+\f c3\sum_{m=0}^9\ve^m \sum_{\substack{i+j+k=m,\\\text{ at least two of } i,j,k \text{ are not zero}}}\CC(\QQ_i,\QQ_j,\QQ_k)\\
&+\frac{c}2\ve^3\sum_{i,j=0}^3\ve^{i+j} \CC(\QQ_R,\QQ_i,\QQ_j)\\
&+c\ve^6\CC(\QQ_R,\QQ_R,\QQ_0+\ve\QQ_1+\ve^2\QQ_2+\ve^3\QQ_3)+c\ve^9\CC(\QQ_R,\QQ_R,\QQ_R).
\end{align*}
Let $\overline{\QQ}^\ve=\QQ_1+\ve\QQ_2+\ve^2\QQ_3$. We introduce the notations:
\begin{align*}
\BB_1&=-\f b2\BB(\QQ_1,\QQ_1)+c\CC(\QQ_0,\QQ_1,\QQ_1),\\
\BB_2&=-b\BB(\QQ_1,\QQ_2)+2c\CC(\QQ_0,\QQ_1,\QQ_2),\\
\overline{\BB}^\ve&=-\f b 2\sum_{\substack{i+j\ge4,1\le i,j\le 3}}\ve^{i+j-4}\BB(\QQ_i,\QQ_j)\\
&\qquad+\f c3\sum_{\substack{i+j+k\ge4, \text{ at least two of } i,j,k \text{ are not zero}}}\ve^{i+j+k-4}\CC(\QQ_i,\QQ_j,\QQ_k).
\end{align*}
Then we obtain the following expansion of $\CJ(\QQ^\ve)$ in $\ve$:
\begin{align}
\CJ(\QQ^\ve)=&\CJ(\QQ_0)+\ve\CH_{\QQ_0}(\QQ_1)+\ve^2\big(\CH_{\QQ_0}(\QQ_2)+\BB_1\big)+\ve^3\big(\CH_{\QQ_0}(\QQ_3)+\BB_2\big)\non\\
&+\ve^3\CH_{\QQ_0}(\QQ_R)+\ve^4\CJ_R^\ve,\label{eq:J-expansion}
\end{align}
where
\begin{align*}
\CJ_R^\ve=&\Be-b\BB(\Qe,\QQ_R)+c\CC(\QQ_R,\Qe,\QQ_0)+\frac{c}2\ve\CC(\QQ_R,\Qe,\Qe)\\
&-\frac{b}2\ve^2\BB(\QQ_R,\QQ_R)+c\ve^2\CC(\QQ_R,\QQ_R,\QQ_0+\ve\Qe)+c\ve^5\CC(\QQ_R,\QQ_R,\QQ_R).
\end{align*}

We denote
\beno
&&\HH_0=\CH_{\QQ_0}(\QQ_1)+\CL(\QQ_0),\\
&&\HH_1=\mathcal{L}(\QQ_1)+\CH_{\QQ_0}(\QQ_2)+\BB_1,\\
&&\HH_2=\mathcal{L}(\QQ_2)+\CH_{\QQ_0}(\QQ_3)+\BB_2,\\
&&\DD_i=\frac{1}{2}((\nabla\vv_i)^T+\nabla\vv_i),\quad \BOm_i=\frac{1}{2}(\nabla\vv_i-(\nabla\vv_i)^T).
\eeno
Plugging the expansions (\ref{eq:v-exp})--(\ref{eq:Q-exp}) and (\ref{eq:J-expansion}) into (\ref{eq:BE-ve})--(\ref{eq:BE-Qe}),
we conclude that\vspace{0.2cm}

{\bf$\bullet$ The order $O(\ve^{-1})$ system}
\begin{align}
\mathcal{J}(\QQ_0)=0.\label{eq:order-1}
\end{align}

{\bf$\bullet$ The order $O(1)$ system}
\begin{align}\label{expan-BE-v0}
&\frac{\partial \vv_0}{\partial t}+\vv_0\cdot\nabla\vv_0=-\nabla p_0
+\nabla\cdot\big(\eta\DD_0+S_{\QQ_0}(\HH_0) -\QQ_0\cdot\HH_0
+\HH_0\cdot\QQ_0+\sigma^d(\QQ_0,\QQ_0)\big),\quad \\
&\nabla\cdot\vv_0=0,\\\label{expan-BE-Q0}
&\frac{\partial \QQ_0}{\partial t}+\vv_0\cdot\nabla \QQ_0 +\QQ_0\cdot\BOm_0-\BOm_0\cdot\QQ_0
=-\frac{1}{\Gamma}\big(\CH_{\QQ_0}(\QQ_1)+\mathcal{L}(\QQ_0)\big) +S_{\QQ_0}(\DD_0),
\end{align}

{\bf$\bullet$ The order $O(\ve)$ system}

\begin{align}\nonumber
\frac{\partial \vv_1}{\partial t}+\vv_0\cdot\nabla\vv_1=&-\vv_1\cdot\nabla\vv_0-\nabla p_1+\nabla\cdot\Big(\eta\DD_1+S_{\QQ_0}(\HH_1)\\\nonumber
&+\xi\big(\BB(\QQ_1, \HH_0)-2\QQ_1(\HH_0:\QQ_0)-2\QQ_0(\HH_0:\QQ_1)\big)-\QQ_1\cdot \HH_0\\\label{expan-BE-v1}
&+\HH_0\cdot\QQ_1-\QQ_0\cdot \HH_1+\HH_1\cdot\QQ_0+\sigma^d(\QQ_1,\QQ_0)+\sigma^d(\QQ_0,\QQ_1)\Big), \\
\nabla\cdot\vv_1=&~0,\\\nonumber
\frac{\partial \QQ_1}{\partial t}+\vv_0\cdot\nabla \QQ_1 =
&-\frac{1}{\Gamma}\big(\mathcal{L}(\QQ_1)+\CH_{\QQ_0}(\QQ_2)+\BB_1\big)
+S_{\QQ_0}\DD_1+\xi\Big(\BB(\DD_0,\QQ_1)\\\nonumber
&-\QQ_1(\QQ_0:\DD_0)-2\QQ_0(\QQ_1:\DD_0)\Big)+\BOm_1\cdot\QQ_0+\BOm_0\cdot\QQ_1\\
&-\QQ_0\cdot\BOm_1
-\QQ_1\cdot\BOm_0-\vv_1\cdot\nabla\QQ_0.\label{expan-BE-Q1}
\end{align}

{\bf$\bullet$ The order $O(\ve^2)$ system}

\begin{align}\nonumber
\frac{\partial \vv_2}{\partial t}+\vv_0\cdot\nabla\vv_2=&-\vv_2\cdot\nabla\vv_0-\vv_1\cdot\nabla
\vv_1-\nabla p_2+\nabla\cdot\Big(\eta\DD_2+S_{\QQ_0}(\HH_2)+\xi\big(\BB(\QQ_1,\HH_1)\\\nonumber
&+\BB(\QQ_2,\HH_0)-2\QQ_1(\HH_1:\QQ_0)-2\QQ_2(\HH_1:\QQ_0)-2\QQ_0(\HH_1:\QQ_1)\\\nonumber
&-2\QQ_1(\HH_0:\QQ_1)\big)-\QQ_0\cdot \HH_2+\HH_2\cdot\QQ_0-\QQ_1\cdot \HH_1-\QQ_2\cdot \HH_0\\\label{expan-BE-v2}
&+\HH_1\cdot\QQ_1+\HH_0\cdot\QQ_2+\sigma^d(\QQ_2,\QQ_0)+\sigma^d(\QQ_1,\QQ_1)+\sigma^d(\QQ_0,\QQ_2)\Big), \\
\nabla\cdot\vv_1=&0,\\\nonumber
\frac{\partial \QQ_2}{\partial t}+\vv_0\cdot\nabla \QQ_2 =
&-\frac{1}{\Gamma}\big(\mathcal{L}(\QQ_2)+\CH_{\QQ_0}(\QQ_3)+\BB_2\big) +S_{\QQ_0}(\DD_2)+\xi\Big(\BB(\DD_0,\QQ_2)+\BB(\DD_1,\QQ_1)\\\nonumber
&-2\QQ_2(\QQ_0:\DD_0)-2\QQ_1(\QQ_1:\DD_0+\DD_1:\QQ_0)-2\QQ_0(\QQ_2:\DD_0+\QQ_1\cdot\DD_1)\Big)\\
&+\BOm_2\cdot\QQ_0+\BOm_0\cdot\QQ_2+\BOm_1\cdot\QQ_1-\QQ_0\cdot\BOm_2
-\QQ_2\cdot\BOm_0-\QQ_1\cdot\BOm_1\nonumber\\
&-\vv_2\cdot\nabla\QQ_0-\vv_1\cdot\nabla\QQ_1.\label{expan-BE-Q2}
\end{align}

\subsection{Derivation of the Ericksen-Leslie system}

Thanks to $\CJ(\QQ_0)=0$ and Proposition \ref{prop:critical ponit}, $\QQ_0(t,\xx)$ takes the form
\ben\label{eq:Q0}
\QQ_0(t,\xx)=s\big(\nn(t,\xx)\nn(t,\xx)-\frac{1}{3}\II\big),
\een
for some $\nn(t,\xx)\in \BS$ and we take $s=\frac {b+\sqrt{b^2+24ac}} {4c}.$
For the sake of simplicity, we denote $\CH_{\QQ_0}$ by $\CH_{\nn}$ in the sequel.
We will prove

\begin{proposition}\label{prop:EL}
If $(\vv_0,\QQ_0)$ is a smooth solution of the system (\ref{expan-BE-v0})--(\ref{expan-BE-Q0}), then $(\nn, \vv_0)$ is necessary a
solution of the Ericksen-Leslie system (\ref{eq:EL-v})--(\ref{eq:EL-n}) with the elastic constants  given by
\begin{align}\label{OF-LD-relation}
k_1=k_3=(2L_1+L_2+L_3)s^2,\quad k_2=2L_1s^2,\quad k_4=L_3s^2,
\end{align}
and the Leslie coefficients given by
\begin{eqnarray}
\left\{
\begin{split}
&\gamma_1=2\Gamma s^2,\quad \gamma_2=-\frac{2\Gamma\xi s(s+2)}{3},\\
&\alpha_1=-\frac{2\Gamma\xi^2 s^2(3-2s)(1+2s)}{3},\quad
\alpha_2=\Gamma s^2-\frac{\Gamma\xi s(2+s)}{3},\quad
\alpha_3=-\Gamma s^2-\frac{\Gamma\xi s(2+s)}{3},\\
&\alpha_4=\eta+\frac{4\Gamma\xi^2(1-s)^2}{9},\quad
\alpha_5=\Gamma\frac{\xi^2s(4-s)}{3}-\Gamma\frac{\xi s(2+s)}{3},\quad
\alpha_6=\Gamma\frac{\xi^2s(4-s)}{3}+\Gamma\frac{\xi s(2+s)}{3}.
\end{split}\right.\label{Leslie-coef}
\end{eqnarray}
\end{proposition}

\begin{Remark}\label{rem:dissipation}
The constants $k_i$ can also be obtained by computing $F_e(\QQ_0)$. Furthermore, it is easy to find
that the Leslie coefficients satisfy the Parodi's relation (\ref{Leslie relation}) and (\ref{Leslie-coeff}).
On the other hand,  it can be verified that  $\alpha_4>0$ and
\begin{align*}
2\alpha_4+\alpha_5+\alpha_6-\frac{\gamma_2^2}{\gamma_1}&=2\eta+\frac{2\Gamma\xi^2}{9}\Big(4(1-s)^2+3s(4-s)-(s+2)^2\Big)=2\eta>0,\\
\frac{3}{2}\alpha_4+\alpha_5+\alpha_6+\alpha_1&=\frac32\eta+\frac{2\Gamma\xi^2}{3}
\Big((1-s)^2+s(4-s)-s^2(3-2s)(1+2s)\Big)
\\&=\frac32\eta+\frac{2}{3}\Gamma\xi^2(1-s)^2(1+2s)^2>0,
\end{align*}
which will ensure that  the energy of the Ericksen-Leslie system is dissipated. 
It should be noticed that here the dissipation coefficient $\alpha_5+\alpha_6-\frac{\gamma_2^2}{\gamma_1}$ 
in (\ref{EL_energy_law}) is strictly negative when $s\neq 1$. 
\end{Remark}

Proposition \ref{prop:EL} will follow from the following two lemmas.

\begin{lemma} \label{lem:EL-n}
Let $\NN=\frac{\partial\nn}{\partial t}+\vv_0\cdot\nabla\nn-\BOm_0\cdot\nn$.
Then $\nn$ satisfies
\beno
\nn\times\big(\hh-\gamma_1\NN-\gamma_2\DD_0\cdot\nn\big)=0,
\eeno
where $\hh=-\frac{\delta E(\nn,\nabla\nn)}{\delta\nn}$ and $E(\nn,\nabla\nn)$ is the Oseen-Frank energy  with the coefficients
given by (\ref{OF-LD-relation}).
\end{lemma}

\begin{proof}
Since  $\CH_\nn(\QQ_1)\in\QO$ by Proposition \ref{prop:H-lower}, it follows from (\ref{expan-BE-Q0}) that
\begin{align*}
\Big(\frac{\partial \QQ_0}{\partial t}+\vv_0\cdot\nabla \QQ_0 +\QQ_0\cdot\BOm_0
-\BOm_0\cdot\QQ_0+\frac1\Gamma\mathcal{L}(\QQ_0)-S_{\QQ_0}(\DD_0)\Big):(\nn\nn^\bot+\nn^\bot\nn)=0.
\end{align*}
Using (\ref{eq:Q0}), we get by some tedious computations that
\begin{align*}
&\frac{\partial \QQ_0}{\partial t}:(\nn\nn^\bot+\nn^\bot\nn)=s(\nn\frac{\partial\nn}{\partial t}
+\frac{\partial\nn}{\partial t}\nn):(\nn\nn^\bot+\nn^\bot\nn)=2s\frac{\partial\nn}{\partial t}\cdot\nn^\bot,\\
&\vv_0\cdot\nabla \QQ_0:(\nn\nn^\bot+\nn^\bot\nn)=s\big(\nn(\vv_0\cdot\nabla\nn)
+(\vv_0\cdot\nabla\nn)\nn\big):(\nn\nn^\bot+\nn^\bot\nn)=2s(\vv_0\cdot\nabla\nn)\cdot\nn^\bot,\\
&\big(\QQ_0\cdot\BOm_0-\BOm_0\cdot\QQ_0\big):(\nn\nn^\bot+\nn^\bot\nn)=s(\nn\nn\cdot\BOm_0-\BOm_0\cdot\nn\nn)
:(\nn\nn^\bot+\nn^\bot\nn)=-2s(\BOm_0\cdot\nn)\cdot\nn^\bot.
\end{align*}
Using (\ref{eq:Q0}) again, we rewrite $S_{\QQ_0}(\DD_0)$ as
\begin{align*}
S_{\QQ_0}(\DD_0)=&\xi\Big(\DD_0\cdot(s\nn\nn+\frac{1-s}{3}\II)+(s\nn\nn+\frac{1-s}{3}\II)
\cdot\DD_0-2s(s\nn\nn+\frac{1-s}{3}\II)(\nn\nn:\DD_0)\Big)\\
=&\xi\Big(s(\nn\DD_0\cdot\nn+\DD_0\cdot\nn\nn)+\frac{2(1-s)}{3}\DD_0
-2s(s\nn\nn+\frac{1-s}{3}\II)(\nn\nn:\DD_0)\Big),
\end{align*}
from which, it follows that
\begin{align*}
S_{\QQ_0}(\DD_0):(\nn\nn^\bot+\nn^\bot\nn)=&\xi\Big(2s(\DD_0\cdot\nn)\cdot\nn^\bot
+\frac{4(1-s)}{3}(\DD_0\cdot\nn)\cdot\nn^\bot\Big)\\
=&\frac{2\xi(2+s)}{3}(\DD_0\cdot\nn)\cdot\nn^\bot.
\end{align*}
Moreover, we have
\begin{align*}
&-\CL(\QQ_0):(\nn\nn^\bot+\nn^\bot\nn)\\
&=2\Big(L_1s\Delta(n_kn_l)+\frac12(L_2+L_3)\big[s(n_kn_m)_{,ml}
+s(n_ln_m)_{,mk}-\frac23\delta_{kl}s(n_in_j)_{,ij}\big]\Big)n_kn^\bot_l\\
&=2s\Big(L_1(\Delta n_kn_l+2n_{k,i}n_{l,i}+n_k\Delta n_l) +\frac12(L_2+L_3)\big[n_{k,ml}n_m+n_{k,m}n_{m,l}
+n_{k,l}n_{m,m}\\
&\quad+n_{k}n_{m,ml}+n_{l,mk}n_m+n_{l,m}n_{m,k}+n_{l,k}n_{m,m}+n_{l}n_{m,km}\big]\Big)n_kn^\bot_l\\
&=2s\Big(L_1\Delta n_l+\frac12(L_2+L_3)\big[n_{k,ml}n_mn_k+n_{m,ml}+n_{l,mk}n_mn_k+n_{l,m}n_{m,k}n_k+n_kn_{l,k}n_{m,m}\big]\Big)n^\bot_l\\
&=2s\Big(L_1\Delta n_l+\frac12(L_2+L_3)\big[-n_{k,m}n_mn_{k,l}+n_{m,ml}+\partial_m(n_kn_{l,k}n_{m})\big]\Big)n^\bot_l.
\end{align*}
On the other hand, we have
\begin{align*}
(\hh)_i=\Big(-\frac{\delta}{\delta\nn}E(\nn,\nabla\nn)\Big)_i
=&{k_1}\partial_i(\partial_jn_j) +{k_2} \big(\Delta n_i-\partial_i(\partial_jn_j)+
\partial_in_kn_l\partial_ln_k-\partial_k(n_kn_l\partial_ln_i)\big)\\
&+ {k_3} (-\partial_in_kn_l\partial_ln_k+\partial_k(n_kn_l\partial_ln_i))\\
=&{k_2}\Delta n_i+(k_1-k_2)\partial_i(\partial_jn_j)+(k_3-k_2)(-\partial_in_kn_l\partial_ln_k+\partial_k(n_kn_l\partial_ln_i))\\
=&s^2\Big(2L_1\Delta n_i+(L_2+L_3)\big[-n_{k,m}n_mn_{k,i}+n_{m,mi}+\partial_m(n_kn_{i,k}n_{m})\big]\Big).
\end{align*}
This means that
\begin{align}
-\CL(\QQ_0):(\nn\nn^\bot+\nn^\bot\nn)=\frac1s\hh\cdot\nn^\bot.\non
\end{align}

Summing up the above identities, we conclude that
\begin{align}
\nn^\bot\cdot\Big(2s^2\NN-\frac{1}{\Gamma}\hh-\frac{2\xi s(2+s)}{3}\DD_0\cdot\nn\Big)=0.\non
\end{align}
The lemma follows by the definition of $\gamma_1$ and $\gamma_2$.
\end{proof}

\begin{lemma}
It holds that
\beno
&&\sigma^L=\eta\DD_0+S_{\QQ_0}(\HH_0) -\QQ_0\cdot\HH_0+\HH_0\cdot\QQ_0,\\
&&\sigma^E=\sigma^d(\QQ_0,\QQ_0),
\eeno
where the coefficients of $\sigma^L$ and $\sigma^E$ are given by (\ref{OF-LD-relation}) and (\ref{Leslie-coef}).
\end{lemma}

\begin{proof}
The key point is to calculate
$$
\HH_0=\CH_\nn(\QQ_1)+\CL(\QQ_0).
$$
By (\ref{eq:Q0}) and the definition of $\QI$, it is easy to see that
\begin{align*}
\frac{\partial \QQ_0}{\partial t}+\vv_0\cdot\nabla \QQ_0 +\QQ_0\cdot\BOm_0-\BOm_0\cdot\QQ_0=s(\nn\NN+\NN\nn)\in\QI.
\end{align*}
Then by (\ref{expan-BE-Q0}), we get
\begin{align*}
\CH_\nn(\QQ_1)=&\CPo\big(-\mathcal{L}(\QQ_0)+\Gamma S_{\QQ_0}(\DD_0)\big).
\end{align*}
We can see from the proof of Lemma \ref{lem:EL-n} that
\beno
S_{\QQ_0}(\DD_0)\cdot\nn=\frac{2\xi(2+s)}{3}(\DD_0\cdot\nn),\quad \hh=-2s\CL(\QQ_0)\cdot\nn,
\eeno
from which and (\ref{def:Pin}), we infer that
\begin{align*}
\CH_\nn(\QQ_1)=&\CPo\big(-\mathcal{L}(\QQ_0)+\Gamma S_{\QQ_0}(\DD_0)\big)\\
=&-\CL(\QQ_0)-\frac1s\hh\nn-\frac1s\nn\hh+\frac2s(\hh\cdot\nn)
\nn\nn+\Gamma\Big(S_{\QQ_0}(\DD_0)\\
&-\frac{\xi(2+s)}{3}(\nn\DD_0\cdot\nn+\DD_0\cdot\nn\nn)
+\frac{2\xi(2+s)}{3}\nn\nn(\DD_0:\nn\nn)\Big),
\end{align*}
which along with Lemma \ref{lem:EL-n} gives
\begin{align}
\CH_\nn(\QQ_1)+\CL(\QQ_0)=-\Gamma s(\NN\nn+\nn\NN)+\Gamma S_{\QQ_0}(\DD_0).\label{eq:H0}
\end{align}

Using (\ref{eq:Q0}), we rewrite $S_{\QQ_0}(\cdot)$ as
\begin{align*}
S_{\QQ_0}(\MM)=&\xi(s\nn\nn+\frac{1-s}{3}\II)\cdot\MM+\MM\cdot(s\nn\nn+\frac{1-s}{3}\II)-2(s\nn\nn+\frac{1-s}{3}\II)s(\MM:\nn\nn-\frac13\tr\MM).
\end{align*}
Then we obtain
\begin{align*}
S_{\QQ_0}(\NN\nn+\nn\NN)=&\xi\Big((s\nn\nn+\frac{1-s}{3}\II)\cdot(\NN\nn+\nn\NN)+(\NN\nn+\nn\NN)\cdot(s\nn\nn+\frac{1-s}{3}\II)\Big)\\
=&\xi s(\nn\NN+\NN\nn)+\frac{2(1-s)}{3}\xi(\NN\nn+\nn\NN)=\frac{2+s}{3}\xi(\nn\NN+\NN\nn),\\
S_{\QQ_0}(\DD_0)=&\xi\Big(s(\nn\DD_0\cdot\nn+\DD_0\cdot\nn\nn)+\frac{2(1-s)}{3}\DD_0
-2s(s\nn\nn+\frac{1-s}{3}\II)(\nn\nn:\DD_0)\Big).
\end{align*}
Then it can be deduced that
\begin{align*}
(\QQ_0+\frac13\II)\cdot S_{\QQ_0}(\DD_0)
=&\xi(s\nn\nn+\frac{1-s}{3}\II)\cdot\Big(s(\nn\DD_0\cdot\nn+\DD_0\cdot\nn\nn)\\
&+\frac{2(1-s)}{3}\DD_0-2s(s\nn\nn+\frac{1-s}{3}\II)(\nn\nn:\DD_0)\Big)\\
=&\xi\Big(s\nn\DD_0\cdot\nn+\frac{s(1-s)}{3}\DD_0\cdot\nn\nn+\frac{2(1-s)^2}{9}\DD_0\\
&-\frac{s^2(1+2s)}{3}\nn\nn(\nn\nn:\DD_0)-\frac{2s(1-s)^2}{9}\II(\nn\nn:\DD_0)\Big),
\end{align*}
and
\begin{align*}
(\QQ_0+\frac13\II): S(\DD_0,\QQ_0)
=&\xi\Big(\frac{s(4-s)}{3}-\frac{s^2(1+2s)}{3}-\frac{2s(1-s)^2}{3}\Big)(\nn\nn:\DD_0)\\
=&\xi\frac{2s(1-s)(1+2s)}{3}(\nn\nn:\DD_0).
\end{align*}
Hence, we get
\begin{align*}
S_{\QQ_0}\big(S_{\QQ_0}(\DD_0)\big)=&\xi^2\Big(\frac{s(4-s)}{3}(\nn\DD_0\cdot\nn+\DD_0\cdot\nn\nn)
+\frac{4(1-s)^2}{9}\DD_0-\frac{2s^2(1+2s)}{3}\nn\nn(\nn\nn:\DD_0)\\
&-\frac{4s(1-s)^2}{9}\II(\nn\nn:\DD_0)
-2\frac{2s(1-s)(1+2s)}{3}(\nn\nn:\DD_0)(s\nn\nn+\frac{1-s}{3}\II)\Big)\\
=&\xi^2\Big(\frac{s(4-s)}{3}(\nn\DD_0\cdot\nn+\DD_0\cdot\nn\nn)
+\frac{4(1-s)^2}{9}\DD_0\\
&-\frac{2s^2(3-2s)(1+2s)}{3}\nn\nn(\nn\nn:\DD_0)-\frac{8s(1+s)^2(1-s)^2}{9}\II(\nn\nn:\DD_0)\Big).
\end{align*}
This gives by (\ref{eq:H0}) that
\begin{align*}
S_{\QQ_0}(\HH_0)&=-\frac{\Gamma\xi s(2+s)}{3}(\NN\nn+\nn\NN)
+\Gamma\xi^2\Big(\frac{s(4-s)}{3}(\nn\DD_0\cdot\nn+\DD_0\cdot\nn\nn)+\frac{4(1-s)^2}{9}\DD_0\\
&\quad-\frac{2s^2(3-2s)(1+2s)}{3}
\nn\nn(\nn\nn:\DD_0)-\frac{8s(1+s)^2(1-s)^2}{9}\II(\nn\nn:\DD_0)\Big).
\end{align*}
On the other hand, we have
\begin{align*}
\HH_0\cdot\QQ_0-\QQ_0\cdot\HH_0
&=s\Gamma\Big(- s(\NN\nn+\nn\NN)+ S_{\QQ_0}(\DD_0)\Big)\cdot(\nn\nn-\frac13\II)\\
&\quad-s\Gamma(\nn\nn-\frac13\II)\cdot\Big(- s(\NN\nn+\nn\NN)+ S_{\QQ_0}(\DD_0)\Big)\\
&=\Gamma s^2(\nn\NN-\NN\nn)-\Gamma\frac{\xi s(2+s)}{3}(\nn\DD_0\cdot\nn-\DD_0\cdot\nn\nn).
\end{align*}
Thus, we conclude that
\begin{align*}
&\eta\DD_0+S_{\QQ_0}(\HH_0) -\QQ_0\cdot\HH_0+\HH_0\cdot\QQ_0\\
&=\alpha_1(\nn\nn:\DD_0)\nn\nn+\alpha_2\nn\NN+\alpha_3\NN\nn+\alpha_4\DD_0
+\alpha_5\nn\nn\cdot\DD_0+\alpha_6\DD_0\cdot\nn\nn+ \text{pressure terms},
\end{align*}
with $\al_i$ given by (\ref{Leslie-coef}).

For the distortion stress, we have
\begin{align*}
\sigma^d_{ij}(\QQ_0,\QQ_0)&=-(L_1Q_{0kl,j}Q_{0kl,i}+L_2Q_{0kl,l}Q_{0kj,i}+L_3Q_{0kj,l}Q_{0kl,i})\\
&=-\big(L_1s^2(n_kn_l)_{,j}(n_kn_l)_{,i}+L_2s^2(n_kn_l)_{,l}(n_kn_j)_{,i}+L_3s^2(n_kn_j)_{,l}(n_kn_l)_{,i}\big)\\
&=-\big(2L_1s^2n_{k,j}n_{k,i}+L_2s^2(n_{k,l}n_{k,i}n_ln_j+n_{l,l}n_{j,i})+L_3s^2(n_{k,l}n_{k,i}n_ln_j+n_{j,l}n_{l,i})\big).
\end{align*}
Using the following facts
\begin{align*}
&(\na\cdot\nn)^2=(\partial_in_i)^2,\quad \big(\nn\cdot(\na\times\nn)\big)^2=\partial_in_j\partial_in_j-\partial_in_j\partial_jn_i-n_in_k\partial_in_j\partial_kn_j,\\
& |\nn\times(\na\times \nn)|^2=n_in_k\partial_in_j\partial_kn_j,\quad \textrm{tr}(\na\nn)^2-(\na\cdot\nn)^2=\partial_in_j\partial_jn_i-(\partial_in_i)^2,
\end{align*}
we infer that
\begin{align*}
&\frac{\partial E(\nn,\nabla\nn)}{\partial n_{k,j}}=k_1\delta_{kj}\partial_ln_l
+k_2(\partial_jn_{k}-\partial_kn_{j}-n_in_j\partial_in_k)+k_3n_in_j\partial_in_k
+(k_2+k_4)(\partial_kn_{j}-\delta_{kj}\partial_ln_l),
\end{align*}
hence,
\begin{align*}
&\frac{\partial E(\nn,\nabla\nn)}{\partial n_{k,j}}n_{k,i}=k_1\partial_in_{j}\partial_ln_l
+k_2\partial_in_{k}(\partial_jn_{k}-\partial_kn_{j}-n_ln_j\partial_ln_k)+k_3n_ln_j\partial_in_{k}\partial_ln_k\\
&\qquad\qquad\qquad\qquad+(k_2+k_4)(\partial_in_{k}\partial_kn_{j}-\partial_in_{j}\partial_ln_l))\\
&\qquad\qquad\qquad=2L_1s^2n_{k,j}n_{k,i}+L_2s^2(n_{k,l}n_{k,i}n_ln_j+n_{l,l}n_{j,i})+L_3s^2(n_{k,l}n_{k,i}n_ln_j+n_{j,l}n_{l,i})\\
&\qquad\qquad\qquad=-\sigma^d_{ij}(\QQ_0,\QQ_0),
\end{align*}
which means that $\sigma^d(\QQ_0,\QQ_0)$ is the same as the Ericksen stress $\sigma^E$.
\end{proof}

\subsection{Existence of Hilbert expansion}
Let $(\vv_0, \nn)$ be a solution of (\ref{eq:EL-v})--(\ref{eq:EL-n}) on $[0,T]$ and satisfy
\ben
\vv_0\in C([0,T];H^{k}), \quad \nabla\nn\in C([0,T];H^{k})\quad \textrm{for} \quad k\ge 20.\label{eq:vn-infor}
\een
Hence, $\QQ_0\in C([0,T];H^{k+1})$ by (\ref{eq:Q0}).

We write $\QQ_1=\QQ_1^\top+\QQ_1^\bot$ with $\QQ_1^\top\in \QI$ and $\QQ_1^\bot\in \QO$.
By Proposition \ref{prop:H-lower}, Proposition \ref{prop:H-inverse} and (\ref{eq:H0}), we get
\ben
\QQ_1^\bot=\CH_{\nn}^{-1}\big(-\CL(\QQ_0)-\Gamma s(\NN\nn+\nn\NN)+\Gamma S_{\QQ_0}(\DD_0)\big)\in C([0,T];H^{k-1}).\label{eq:QQ-out}
\een

Next we solve $(\vv_1,\QQ_1^\top)$. Let us first derive the equations of $(\vv_1,\QQ_1^\top)$.
We denote by $L(\cdot)$ the linear function with the coefficients belonging to $C([0,T];H^{k-1})$,
and by $R\in C([0,T];H^{k-3})$ some function depending only on $\nn,\vv_0, \QQ^\bot$. Set
\begin{align*}
\overline{\BB}_1(\QQ, \widetilde\QQ)&=-b\Big(\QQ\cdot\widetilde\QQ-\frac{1}{3}(\QQ:\widetilde\QQ)\II\Big)
+c\big(2(\QQ:\QQ_0)\widetilde\QQ+(\QQ:\widetilde\QQ)\QQ_0\big).
\end{align*}
Then $\BB_1$ can be written as
\begin{align*}
\BB_1=\overline{\BB}_1(\QQ_1,\QQ_1)&=\overline{\BB}_1(\QQ_1^{\top},\QQ_1^{\top})
+\overline{\BB}_1(\QQ_1^{\top},\QQ_1^{\bot})
+\overline{\BB}_1(\QQ_1^{\bot},\QQ_1^{\top})+\overline{\BB}_1(\QQ_1^{\bot},\QQ_1^{\bot})\\
&=\overline{\BB}_1(\QQ_1^{\top},\QQ_1^{\top})+L(\QQ_1^{\top}, \vv_1).
\end{align*}
It is easy to show that
\ben
\overline{\BB}_1(\QQ_1^{\top},\QQ_1^{\top})\in \QO.\label{eq:B1-out}
\een

\begin{lemma}\label{lem:Pin}
It holds that
\begin{align*}
&\CPo\big(\frac{\partial \QQ_1}{\partial t}+\vv_0\cdot\nabla\QQ_1\big)=L(\QQ_1^{\top})+R,\\
&\CPi\big(\frac{\partial \QQ_1}{\partial t}+\vv_0\cdot\nabla\QQ_1\big)=
\frac{\partial \QQ_1^{\top}}{\partial t}+\vv_0\cdot\nabla\QQ_1^{\top}+L(\QQ_1^{\top})+R.
\end{align*}
\end{lemma}

\begin{proof}
Assume that $\QQ_1^{\top}=\nn\nn^\bot+\nn^\bot\nn$ with $\nn^\bot\cdot\nn=0$. Then we have
\begin{align*}
\frac{\partial \QQ_1^{\top}}{\partial t}+\vv_0\cdot\nabla \QQ_1^{\top}=\nn\dot\nn^\bot+\dot\nn\nn^\bot+\dot\nn^\bot\nn+\nn^\bot\dot\nn,
\end{align*}
where $\dot\mm=\partial_t\mm+\vv_0\cdot\nabla\mm$. Using the facts that
\beno
\nn^\bot\cdot\nn=\dot\nn\cdot\nn=0,\quad \dot\nn^\bot\cdot\nn+\nn^\bot\cdot\dot\nn
=(\partial_t+\vv_0\cdot\nabla)(\nn^\bot\cdot\nn)=0,
\eeno
we get
\begin{align*}
(\II-\nn\nn)\cdot\Big(\frac{\partial \QQ_1^{\top}}{\partial t}+\vv_0\cdot\nabla \QQ_1^{\top}\Big)\cdot\nn=\dot\nn^\bot+(\nn^\bot\cdot\dot\nn)\nn.
\end{align*}
Then we infer from (\ref{def:Pin}) that
\begin{align*}
\CPi\big(\frac{\partial \QQ_1^{\top}}{\partial t}+\vv_0\cdot\nabla\QQ_1^{\top}\big)&=\nn\big(\dot\nn^\bot
+(\nn^\bot\cdot\dot\nn)\nn\big)+\big(\dot\nn^\bot+(\nn^\bot\cdot\dot\nn)\nn\big)\nn\\
&=\nn\dot\nn^\bot+\dot\nn^\bot\nn+L(\QQ_1^\top),
\end{align*}
from which, it follows that
\begin{align*}
\CPo\big(\frac{\partial \QQ_1}{\partial t}+\vv_0\cdot\nabla\QQ_1\big)&=
\CPo\big(\frac{\partial \QQ_1^{\top}}{\partial t}+\vv_0\cdot\nabla\QQ_1^{\top}\big)+R\\
&=L(\QQ_1^{\top})+R.
\end{align*}
Hence, we have
\begin{align*}
\CPi\big(\frac{\partial \QQ_1}{\partial t}+\vv_0\cdot\nabla\QQ_1\big)&=\frac{\partial \QQ_1}{\partial t}+\vv_0\cdot\nabla\QQ_1-\CPo\big(\frac{\partial \QQ_1}{\partial t}+\vv_0\cdot\nabla\QQ_1\big)\\
&=\frac{\partial \QQ_1^{\top}}{\partial t}+\vv_0\cdot\nabla\QQ_1^{\top}+L(\QQ_1^{\top})+R.
\end{align*}
The proof is finished.
\end{proof}

We denote
\begin{align*}
&\CA_1=\CPi\big(\mathcal{L}(\QQ_1^{\top})\big),\quad\CA_2=\CPo\big(\mathcal{L}(\QQ_1^{\top})\big),\\
&\mathcal{C}_1=\CPi\big(S_{\QQ_0}\DD_1+\BOm_1\cdot\QQ_0-\QQ_0\cdot\BOm_1\big),\quad
\mathcal{C}_2=\CPo\big( S_{\QQ_0}\DD_1+\BOm_1\cdot\QQ_0-\QQ_0\cdot\BOm_1\big).
\end{align*}
Since $\mathcal{L}(\QQ_1)=\mathcal{L}(\QQ_1^{\top})+R$ and $\CH_{\QQ_0}(\QQ_2)\in \QO$,
we take $\CPo$ on both sides of (\ref{expan-BE-Q1}) and use Lemma \ref{lem:Pin} and (\ref{eq:B1-out}) to get
\begin{align}\label{eq:Q1-in}
&\frac{\partial \QQ_1^{\top}}{\partial t}+\vv_0\cdot\nabla\QQ_1^{\top}
=-\frac1\Gamma\CA_1+\mathcal{C}_1+L(\QQ_1^{\text{in}},\vv_1)+R,
\end{align}
and take $\CPo$ on both sides of (\ref{expan-BE-Q1}) to get
\begin{align*}
-\f1{\Gamma}\big(\CA_2+\CH_{\QQ_0}(\QQ_2)+\overline\BB_1(\QQ_1^{\top},\QQ_1^{\top})\big)+
\mathcal{C}_2+L(\vv_1,\QQ_1^\top)+R=0.
\end{align*}
This also implies
\begin{align}\label{eq:H1}
\HH_1&=\mathcal{L}(\QQ_1)+\CH_{\QQ_0}(\QQ_2)+\BB_1\non\\
&=\CA_1+\Gamma\mathcal{C}_2+L(\vv_1,\QQ_1^\top)+R.
\end{align}

Plugging (\ref{eq:H1}) into (\ref{expan-BE-v1}), we derive the equations of $(\vv_1,\QQ_1^\top)$:
\begin{align}\nonumber
&\frac{\partial \vv_1}{\partial t}+\vv_0\cdot\nabla\vv_1=-\nabla p_1+\nabla\cdot\Big(\eta\DD_1
+S_{\QQ_0}(\CA_1+\Gamma\mathcal{C}_2)-\QQ_0\cdot(\CA_1+\Gamma\mathcal{C}_2)\\
&\qquad\qquad+(\CA_1+\Gamma\mathcal{C}_2)\cdot\QQ_0+\sigma^d(\QQ_1,\QQ_0)+\sigma^d(\QQ_0,\QQ_1)+L(\vv_1,\QQ_1^\top)+R\Big), \label{expan-BE-v11}\\
&\nabla\cdot\vv_1=0,\nonumber\\
&\frac{\partial \QQ_1^{\top}}{\partial t}+\vv_0\cdot\nabla\QQ_1^{\top}
=-\frac1\Gamma\CA_1+\mathcal{C}_1+L(\vv_1,\QQ_1^\top)+R.\label{expan-BE-Q11}
\end{align}

The system (\ref{expan-BE-v11})--(\ref{expan-BE-Q11}) is just a linear system. To see its solvability, we present a priori estimate for the following energy
\beno
E_{k}\eqdefa\sum_{|\ell|=0}^{k-4}\big\langle\pa^\ell\vv_1, \pa^\ell\vv_1\big\rangle
+\big\langle\pa^\ell\QQ_1^{\top}, \CL(\pa^\ell\QQ_1^{\top})\big\rangle
+\langle\QQ_1^{\top}, \QQ_1^{\top}\rangle.
\eeno
We will show that
\begin{align*}
\frac{d}{d t} E_k\le C\big(E_k+\|R(t)\|_{H^{k-3}}\big),
\end{align*}
which will ensure that the system (\ref{expan-BE-v11})--(\ref{expan-BE-Q11}) has a unique solution $(\vv_1,\QQ_1^\top)$ on $[0,T]$
satisfying
\ben
\vv_1\in C([0,T];H^{k-4}),\quad \QQ_1^\top\in C([0,T];H^{k-3}).
\een

In what follows, we give an estimate for the term of $\ell=0$ in $E_k$,
the proof for general case is similar. We set
\beno
E_{1}=\big\langle\vv_1, \vv_1\big\rangle+\big\langle\QQ_1^{\top}, \CL(\QQ_1^{\top})\big\rangle
+\langle\QQ_1^{\top}, \QQ_1^{\top}\rangle.
\eeno
First of all, we get by (\ref{expan-BE-Q11}) and Lemma \ref{lem:L} that for any $\delta>0$,
\begin{align*}
\frac12\frac{d}{d t}\langle\QQ_1^{\top}, \QQ_1^{\top}\rangle
&=\Big\langle-\frac1\Gamma\mathcal{L}(\QQ_1^{\top})+S_{\QQ_0}\DD_1+\BOm_1
\cdot\QQ_0-\QQ_0\cdot\BOm_1, \QQ_1^{\top}\Big\rangle
+\big\langle L(\vv_1,\QQ_1^{\top})+R, \QQ_1^{\top}\big\rangle\\
&\le \delta\|\nabla\vv_1\|_{L^2}^2+C_\delta\|\QQ_1^{\top}\|_{H^1}^2+\|R\|_{L^2}^2.
\end{align*}
Using (\ref{expan-BE-v11})--(\ref{expan-BE-Q11}) again, we get
\begin{align*}
&\frac12\frac{d}{dt}\big(\langle\vv_1, \vv_1\rangle+\langle\QQ_1^{\top}, \CL(\QQ_1^{\top})\rangle\big)
=\langle\partial_t\vv_1, \vv_1\rangle+\langle\partial_t\QQ_1^{\top}, \CL(\QQ_1^{\top})\rangle\\
&=-\eta\langle\DD_1,\DD_1\rangle-\Big\langle
S_{\QQ_0}(\CA_1+\Gamma\mathcal{C}_2)-\QQ_0\cdot(\CA_1+\Gamma\mathcal{C}_2)+(\CA_1+\Gamma\mathcal{C}_2)\cdot\QQ_0\\
&\quad+\sigma^d(\QQ_1,\QQ_0)+\sigma^d(\QQ_0,\QQ_1)+L(\vv_1,\QQ_1^{\top})+R, \nabla\vv_1\Big\rangle-\big\langle\vv_0\cdot\nabla\QQ_1^{\top},\CL(\QQ_1^{\top})\big\rangle\\
&\quad
-\frac1\Gamma\Big\langle\CPi\big(\mathcal{L}(\QQ_1^{\top})\big),\CL(\QQ_1^{\top})\Big\rangle
+\langle\mathcal{C}_1,\CL(\QQ_1^{\top})\rangle+\big\langle L(\vv_1,\QQ_1^{\top})+R,\CL(\QQ_1^{\top})\big\rangle.
\end{align*}
The key point is that we find the following dissipation
\begin{align*}
&-\big\langle S_{\QQ_0}(\CA_1+\Gamma\mathcal{C}_2)-\QQ_0\cdot(\CA_1+\Gamma\mathcal{C}_2)
+(\CA_1+\Gamma\mathcal{C}_2)\cdot\QQ_0, ~\nabla\vv_1\big\rangle+\langle\mathcal{C}_1,~\CL(\QQ_1^{\top})\rangle\\
&=-\big\langle S_{\QQ_0}(\CA_1+\Gamma\mathcal{C}_2),~ \DD_1\big\rangle+\big\langle\QQ_0\cdot(\CA_1+\Gamma\mathcal{C}_2)
-(\CA_1+\Gamma\mathcal{C}_2)\cdot\QQ_0,~ \BOm_1\big\rangle+\langle\mathcal{C}_1,~\CL(\QQ_1^{\top})\rangle\\
&=-\big\langle\CA_1+\Gamma\mathcal{C}_2,~  S_{\QQ_0}\DD_1\big\rangle+\big\langle\CA_1+\Gamma\mathcal{C}_2
,~ \QQ_0\cdot\BOm_1-\BOm_1\cdot\QQ_0\big\rangle+\langle\mathcal{C}_1,~\CL(\QQ_1^{\top})\rangle\\
&=-\big\langle\CA_1+\Gamma\mathcal{C}_2,~  S_{\QQ_0}\DD_1+ \BOm_1\cdot\QQ_0-
\QQ_0\cdot\BOm_1\big\rangle+\langle\mathcal{C}_1,~\CL(\QQ_1^{\top})\rangle\\
&=-\big\langle\CPi\big(\mathcal{L}(\QQ_1^{\top})\big),~  S_{\QQ_0}\DD_1+ \BOm_1\cdot\QQ_0-
\QQ_0\cdot\BOm_1\big\rangle\\
&\quad+\big\langle\CPi\big(S_{\QQ_0}\DD_1+\BOm_1\cdot\QQ_0
-\QQ_0\cdot\BOm_1\big),~\CL(\QQ_1^{\top})\big\rangle\\
&\quad-\Gamma\big\langle\CPo\big(S_{\QQ_0}\DD_1+\BOm_1\cdot\QQ_0
-\QQ_0\cdot\BOm_1\big), ~  S_{\QQ_0}\DD_1+ \BOm_1\cdot\QQ_0-
\QQ_0\cdot\BOm_1\big\rangle\\
&=-\Gamma\big\langle\CPo\big(S_{\QQ_0}\DD_1+\BOm_1\cdot\QQ_0
-\QQ_0\cdot\BOm_1\big), ~  S_{\QQ_0}\DD_1+ \BOm_1\cdot\QQ_0-
\QQ_0\cdot\BOm_1\big\rangle\le 0.
\end{align*}
For the other terms, we have
\begin{align*}
&-\big\langle\sigma^d(\QQ_1,\QQ_0)+\sigma^d(\QQ_0,\QQ_1)+L(\vv_1,\QQ_1^{\top})+R, \nabla\vv_1\big\rangle
+\big\langle L(\vv_1,\QQ_1^{\top})+R,\CL(\QQ_1^{\top})\big\rangle\\
&\le \delta\|\nabla\vv_1\|_{L^2}^2+C_\delta\big(\|\vv_1\|_{L^2}^2+\|\QQ_1^{\top}\|_{H^1}^2+\|R\|_{H^1}\big),
\end{align*}
and for any $\QQ$,
\begin{align*}
&-\big\langle\vv_0\cdot\nabla\QQ,\CL(\QQ)\big\rangle\\
&=\int_\BR v_{0j} Q_{kl,j}\Big(L_1\Delta Q_{kl}
+\frac12(L_2+L_3)(Q_{km,ml}+Q_{lm,mk}-\frac23\delta_{kl}Q_{ij,ij})\Big)\ud\xx\\
&=\int_\BR\Big( - L_1v_{0j} Q_{kl,mj}Q_{kl,m}
-\frac12(L_2+L_3)(v_{0j} Q_{kl,lj}Q_{km,m}+v_{0j}Q_{kl,kj}Q_{lm,m}\\
&\qquad\qquad- L_1v_{0j,m} Q_{kl,j}Q_{kl,m}
-\frac12(L_2+L_3)(v_{0j,l} Q_{kl,j}Q_{km,m}+v_{0j,k} Q_{kl,j}Q_{lm,m}\Big)\ud\xx\\
&=\int_\BR\Big(- L_1v_{0j,m} Q_{kl,j}Q_{kl,m}
-\frac12(L_2+L_3)(v_{0j,l} Q_{kl,j}Q_{km,m}+v_{0j,k} Q_{kl,j}Q_{lm,m}\Big)\ud\xx\\
&\le C\|\QQ\|_{H^1}^2.
\end{align*}
Thus, we get
\begin{align*}
-\big\langle\vv_0\cdot\nabla\QQ_1^{\top},\CL(\QQ_1^{\top})\big\rangle\le C\|\QQ_1^{\top}\|_{H^1}^2.
\end{align*}

Summing up, we obtain
\begin{align*}
\frac{d}{d t} E_1\le C\big(E_1+\|R\|_{H^1}\big).
\end{align*}
This completes the proof of existence of $(\vv_1,\QQ_1)$.

Again, we write $\QQ_2=\QQ_2^\top+\QQ_2^\bot$ with $\QQ_2^\top\in \QI$ and $\QQ_2^\bot\in \QO$. By (\ref{eq:H1}), we
can determine $\QQ_2^\bot$ by
\ben
\QQ_2^\bot=\CH_\nn^{-1}\big(-\mathcal{L}(\QQ_1)-\BB_1+\CA_1+\Gamma\mathcal{C}_2+L(\vv_1,\QQ_1^\top)+R\big)\in C([0,T];H^{k-5}).
\een
Then $(\vv_2,\QQ_2^\top,\QQ_3)$ can be solved in a similar way as $(\vv_1,\QQ_1^\top)$. We left it to the interested readers.

Summing up, we prove

\begin{proposition}\label{prop:Hilbert}
Let $(\vv_0, \nn)$ be a solution of (\ref{eq:EL-v})--(\ref{eq:EL-n}) on $[0,T]$ and satisfy
\beno
\vv_0\in C([0,T];H^{k}), \quad \nabla\nn\in C([0,T];H^{k})\quad \textrm{for} \quad k\ge 20.
\eeno
There exists the solution $(\vv_i, \QQ_i)(i=0,1,2)$ and $\QQ_3\in \QO$ of the system (\ref{expan-BE-v1})--(\ref{expan-BE-Q2}) satisfying
\beno
\vv_i\in  C([0,T];H^{k-4i}), \quad \QQ_i\in C([0,T];H^{k+1-4i})(i=0,1,2),\quad \QQ_3\in C([0,T];H^{k-9}).
\eeno

\end{proposition}

\setcounter{equation}{0}

\section{Uniform estimates for the remainder}

\subsection{The system for the remainder}

In this subsection, we derive the equations for the remainder $(\vv_R^\ve,\QQ_R^\ve)$ in the Hilbert expansion (\ref{eq:v-exp})--(\ref{eq:Q-exp}).
In what follows, we omit the superscript $\ve$ of $(\vv_R^\ve,\QQ_R^\ve)$.

By (\ref{eq:J-expansion}) and the definitions of $\HH_i$(i=0,1,2), the molecular field $\HH(\QQ^\ve)$
can be expanded into
\begin{align*}
\HH(\QQ^\ve)=\frac{1}{\ve}\CJ(\QQ^\ve)+\CL(\QQ^\ve)=&\HH_0+\ve \HH_1+\ve^2\HH_2+\ve^2\HH_R+\ve^3\CJ_R^\ve,
\end{align*}
where $\HH_R=\CH^\ve_\nn(\QQ_R)\triangleq \CH_{\nn}(\QQ_R)+\ve\CL(\QQ_R)$. We denote
\begin{align*}
\tilde{\vv}=\vv_0+\ve\vv_1+\ve^2\vv_2,\quad \wD=\DD_0+\ve\DD_1+\ve^2\DD_2,
\quad\wO=\BOm_0+\ve\BOm_1+\ve^2\BOm_2.
\end{align*}

Thanks to (\ref{eq:order-1})--(\ref{expan-BE-Q2}) and (\ref{eq:BE-ve})--(\ref{eq:BE-Qe}), we obtain
\begin{align}
\label{expan-BE-vR}
\frac{\partial \vv_R}{\partial t}=&-\tilde\vv\cdot\nabla\vv_R-\nabla p_R+\eta\Delta\vv_R
+\nabla\cdot\Big(\frac1\ve S_{\QQ_0}(\HH_R)
-\frac1\ve\QQ_0\cdot \HH_R+\frac1\ve\HH_R\cdot\QQ_0\Big)\\
&\qquad+\nabla\cdot\GG_R+\GG'_R,\non\\
\nabla\cdot\vv_R=&~0,\\\label{expan-BE-QR}
\frac{\partial \QQ_R}{\partial t} =
&-\frac{1}{\Gamma\ve}\CH_{\nn}^\ve(\QQ_R)
+S_{\QQ_0}\DD_R+\BOm_R\cdot\QQ_0-\QQ_0\cdot\BOm_R+\FF_R.
\end{align}
Let us give the precise formulation of $\FF_R, \GG_R, \GG_R'$. The term $\GG_R'$ takes the from
\begin{align}\label{GRp}
\GG_R'=-\vv_1\cdot\nabla\vv_2-\vv_2\cdot\nabla\vv_1-\ve\vv_2\cdot\nabla\vv_2
-\vv_R\cdot\nabla\tv-\ve^3\vv_R\cdot\nabla\vv_R.
\end{align}
The term $\FF_R$ consists of five parts
\ben
\FF_R=\FF_1+\FF_2+\FF_3+\FF_4+\FF_5,
\een
where $\FF_1$ is independent of $(\vv_R, \QQ_R)$:
\begin{align*}
\FF_1=&-\frac1\Gamma\Be+\sum_{i+j\ge3}\ve^{i+j-3}\Big(\xi\BB(\DD_i,\QQ_j)+\BOm_i\cdot\QQ_j-\QQ_j\cdot\BOm_i\Big)\\
&-2\xi\sum_{i+j+k\ge3}\ve^{i+j+k-3}\QQ_i(\DD_j:\QQ_k)-\vv_0\cdot\nabla\QQ_3-\vv_1\cdot\nabla
(\QQ_2+\ve\QQ_3)-\vv_2\cdot\nabla\Qe-\frac{\partial\QQ_3}{\partial t},
\end{align*}
and $\FF_2, \FF_3$ linearly depend on $(\vv_R, \QQ_R)$:
\begin{align*}
\FF_2=&\xi\Big(\BB(\wD, \QQ_R)-2\QQ_R\sum_{i=0}^2\sum_{j=0}^3\ve^{i+j}\DD_i:\QQ_j-2\sum_{i=0}^2
\sum_{j=0}^3\ve^{i+j}\QQ_j(\QQ_R:\DD_i)\Big)+\wO\cdot\QQ_R\nonumber\\
&\quad-\QQ_R\cdot\wO-\frac1\Gamma\Big(-b\BB(\Qe,\QQ_R)+c\CC(\QQ_R,\Qe,\QQ_0)
+\frac{c}2\ve\CC(\QQ_R,\Qe,\Qe)\Big)-\tilde\vv\cdot\nabla \QQ_R,\nonumber\\\nonumber
\FF_3=&-\vv_R\cdot\nabla(\QQ_0+\ve\Qe)-\ve\Qe\cdot\BOm_R
+\ve\BOm_R\cdot\Qe\\\nonumber
&\quad+\xi\Big(\ve\Qe\cdot\DD_R+\ve\DD_R\cdot\Qe
-\frac23\ve\II\Qe:\DD_R+\sum_{i+j\ge1}\ve^{i+j}\QQ_i(\DD_R:\QQ_j)\Big),
\end{align*}
and $\FF_4,\FF_5$ nonlinearly depend on $(\vv_R, \QQ_R)$:
\begin{align*}
\FF_4=&~\ve^3\Big(-\vv_R\cdot\nabla\QQ_R-\QQ_R\cdot\BOm_R+\BOm_R\cdot\QQ_R
+\xi\big[\DD_R\cdot\QQ_R+\QQ_R\cdot\DD_R-\frac23\II(\QQ_R:\DD_R)\\
&\quad-2(\QQ_0+\ve\Qe)(\QQ_R:\DD_R)-2\QQ_R((\QQ_0+\ve\Qe):\DD_R)-2\ve^3\QQ_R(\QQ_R:\DD_R)\big]\Big),\\
\FF_5=&-\frac1\Gamma\big(-b\ve^2\BB(\QQ_R,\QQ_R)+c\ve^2\CC(\QQ_R,\QQ_R,\QQ_0+\ve\Qe)+c\ve^5\CC(\QQ_R,\QQ_R,\QQ_R)\big)\\
&\quad-\xi\ve^3\QQ_R(\QQ_R:\wD).
\end{align*}
Similarly, $\GG_R$ can be written as
\begin{align}\label{GR}
\GG_R=\GG_1+\GG_2+\GG_3+\GG_4,
\end{align}
where $\GG_1$ is given by
\begin{align*}
\GG_1=&~\xi\sum_{i+j\ge3}\ve^{i+j-3}\BB(\QQ_i,\HH_j)-2\xi\sum_{i+j+k\ge3}\ve^{i+j+k-3}\QQ_i(\HH_j:\QQ_k)\\
&+\sum_{i+j\ge3}\ve^{i+j-3}\big(\QQ_i\cdot\HH_j-\HH_j\cdot\QQ_i+\sigma^d(\QQ_i,\QQ_j)\big),\\
\end{align*}
and $\GG_2,\GG_3$ are given by
\begin{align*}
\GG_2=&~\xi\BB(\Qe,\HH_R)
-2\xi\sum_{i+j\ge1}\ve^{i+j-1}\QQ_i(\HH_R:\QQ_j)+\Qe\cdot\HH_R-\HH_R\cdot\Qe\\
&+\xi\sum_{i=0}^2\ve^i\BB(\QQ_R,\HH_i)-2\xi\sum_{i+j\ge1}\ve^{i+j}\big[\QQ_i(\HH_j:\QQ_R)+\QQ_R(\HH_j:\QQ_i)\big]\\
&+\sum_{j=0}^{2}\ve^j[\QQ_R\cdot\HH_j-\HH_j\cdot\QQ_R]+\sigma^d(\QQ_0+\ve\Qe,\QQ_R)+\sigma^d(\QQ_R,\QQ_0+\ve\Qe),\\
\GG_3=&~\xi\BB(\QQ_0+\ve\Qe,\CJ_R^\ve)-2\xi\sum_{i,j=0}^3\ve^{i+j}\QQ_i(\CJ_R^\ve:\QQ_j)+(\QQ_0+\ve\Qe)\cdot\CJ_R^\ve-\CJ_R^\ve\cdot(\QQ_0+\ve\Qe),
\end{align*}
and $\GG_4$ is given by
\begin{align*}
\GG_4=&-2\xi\ve^2(\QQ_0+\ve\Qe+\frac13\II)\big(\QQ_R:(\HH_R+\ve\CJ_R^\ve)\big)-2\xi\ve^2\QQ_R\big((\QQ_0+\ve\Qe):(\HH_R+\ve\CJ_R^\ve)\big)\\
&-2\xi\ve^3\QQ_R(\QQ_R:(\HH_0+\ve \HH_1+\ve^2\HH_2+\ve^2\HH_R+\ve^3\CJ_R^\ve))
\\&+\ve^3(\QQ_R\cdot\CJ_R^\ve-\CJ_R^\ve\cdot\QQ_R)+\ve^3\sigma^d(\QQ_R,\QQ_R).
\end{align*}

\subsection{A key lemma} For $\QQ_1,\QQ_2\in L^2(\BR)^{3\times 3}$, we define the inner product
\begin{align}
\langle\QQ_1, \QQ_2\rangle\eqdefa \int_\BR \QQ_1(\xx):\QQ_2(\xx) \ud\xx.\non
\end{align}

The following lemma plays an important role in the energy estimates.

\begin{lemma} \label{lem:control}
For any $\delta>0$, there exists a constant
$C=C(\delta,\|\nabla_{t,\xx}\nn\|_{L^\infty}, \|\nabla\nn_t\|_{L^\infty})$ such that for
 any $\QQ\in\BQ$, it holds that
\begin{align*}
\frac{1}{\ve}\big\langle\partial_t(\nn\nn)\cdot\QQ,
\QQ\big\rangle\le&~\delta \Big\langle\frac{1}{\ve}\CH_\nn(\QQ)+\CL(\QQ), \frac{1}{\ve}\CH_\nn(\QQ)+\CL(\QQ)\Big\rangle\\
\nonumber&\qquad+C_\delta\Big(\Big\langle\frac{1}{\ve}\CH_\nn(\QQ)+\CL(\QQ), \QQ\Big\rangle+\langle\QQ,\QQ\rangle\Big),\\
\frac{1}{\ve}\big\langle\QQ:\partial_t(\nn\nn),
\QQ:\nn\nn\big\rangle\le&~\delta \Big\langle\frac{1}{\ve}\CH_\nn(\QQ)+\CL(\QQ), \frac{1}{\ve}\CH_\nn(\QQ)+\CL(\QQ)\Big\rangle\\
\nonumber&\qquad+C_\delta\Big(\Big\langle\frac{1}{\ve}\CH_\nn(\QQ)+\CL(\QQ), \QQ\Big\rangle+\langle\QQ,\QQ\rangle\Big).
\end{align*}
\end{lemma}

\begin{proof}
Let $\QQ=\QQ^\bot+\QQ^\top$, where $\QQ^\bot\in\QO$ and $\QQ^\top\in\QI$.  Thus, we have
\begin{align*}
\frac{1}{\ve}\big\langle\partial_t(\nn\nn)\cdot\QQ,
\QQ\big\rangle=\frac{1}{\ve}\big\langle\partial_t(\nn\nn)\cdot\QQ^\top,
\QQ^\top\big\rangle+\frac{2}{\ve}\big\langle\partial_t(\nn\nn)\cdot\QQ^\top,
\QQ^\bot\big\rangle+\frac{1}{\ve}\big\langle\partial_t(\nn\nn)\cdot\QQ^\bot,
\QQ^\bot\big\rangle.
\end{align*}
Thanks to the fact that $\partial_t(\nn\nn)=\nn\nn_t+\nn_t\nn\in\QI$ and Lemma \ref{lem:ker-vanish}, the first term on the right hand side vanishes.
By Proposition \ref{prop:H-lower}, the third term is bounded by
\beno
\|\nn_t\|_{L^\infty}\frac1\ve\|\QQ^\bot\|_{L^2}^2\le C\|\nn_t\|_{L^\infty}\frac{1}{\ve}\langle\CH_\nn(\QQ), \QQ\rangle.
\eeno
For the second term, we infer from Proposition \ref{prop:H-inverse} that
\begin{align*}
\frac{1}{\ve}\big\langle\partial_t(\nn\nn)\cdot\QQ^\top,\QQ^\bot\big\rangle
=&\Big\langle\CH_\nn^{-1}\big(\partial_t(\nn\nn)\cdot\QQ^\top\big),\frac1\ve\CH_\nn\QQ\Big\rangle\\
=&\Big\langle\CH_\nn^{-1}\big(\partial_t(\nn\nn)\cdot\QQ^\top\big),\frac1\ve\CH_\nn\QQ+\CL(\QQ)\Big\rangle
-\big\langle\CH_\nn^{-1}\big(\partial_t(\nn\nn)\cdot\QQ^\top\big),\CL(\QQ)\big\rangle\\
=&C\|\nn_t\|_{L^\infty}\|\QQ^\top\|_{L^2}\Big\|\frac1\ve\CH_\nn\QQ+\CL(\QQ)\Big\|_{L^2}
+C_2\big(\|\nabla\QQ\|_{L^2}^2+\|\QQ\|_{L^2}^2\big),
\end{align*}
where $C_2$ depends on  $\|\nabla_{t,\xx}\nn\|_{L^\infty}$ and $\|\nabla\nn_t\|_{L^\infty}$. This gives
the first inequality by Lemma \ref{lem:L}.

Similarly, we have
\beno
\frac{1}{\ve}\big\langle\QQ:\partial_t(\nn\nn),
\QQ:\nn\nn\big\rangle=\frac{1}{\ve}\big\langle\QQ^\top:\partial_t(\nn\nn),
\QQ^\bot:\nn\nn\big\rangle+\f 1 \ve\big\langle\QQ^\bot:\partial_t(\nn\nn),
\QQ^\bot:\nn\nn\big\rangle.
\eeno
Then the second inequality follows in the same way.
\end{proof}

\subsection{Uniform energy estimates}
Throughout this subsection, we assume that $\vv_i\in C([0,T];H^{k-4i})$ for $i=0,1,2$ and $\QQ_i\in C([0,T];H^{k+1-4i})$ for $i=0, 1, 2, 3$.
We denote by $C$ a constant depending on $\displaystyle\sum_{i=0}^2\sup_{t\in [0,T]}\|\vv_i(t)\|_{H^{k-4i}}$
and $\displaystyle\sum_{i=0}^3\sup_{t\in [0,T]}\|\QQ_i(t)\|_{H^{k+1-4i}}$, and independent of $\ve$.

We introduce the following energy functional
\begin{align*}\nonumber
\mathfrak{E}(t)\eqdefa& \f12\int\Big(|\vv_R|^2+\frac1\ve\CH_\nn^\ve(\QQ_R):\QQ_R+|\QQ_R|^2\Big)
+\ve^2\Big(|\nabla\vv_R|^2+\frac1\ve\CH_\nn^\ve(\nabla\QQ_R):\nabla\QQ_R\Big)\\
&\quad+\ve^4\Big(|\Delta\vv_R|^2+\frac1\ve\CH_\nn^\ve(\Delta\QQ_R):\Delta\QQ_R\Big)\ud\xx, \\\nonumber
\mathfrak{F}(t)\eqdefa &\int\Big(\eta|\nabla\vv_R|^2+\frac1{\Gamma\ve^2}\CH_\nn^\ve(\QQ_R):\CH_\nn^\ve(\QQ_R)\Big)+
\ve^2\Big(\eta|\Delta\vv_R|^2+\frac1{\Gamma\ve^2}\CH_\nn^\ve(\nabla\QQ_R):\CH_\nn^\ve(\nabla\QQ_R)\Big)\\
&\quad+\ve^4\Big(\eta|\nabla\Delta\vv_R|^2+\frac1{\Gamma\ve^2}\CH_\nn^\ve(\Delta\QQ_R):\CH_\nn^\ve(\Delta\QQ_R)\Big)\ud\xx.
\end{align*}

The uniform energy estimate is stated as follows.

\begin{proposition}\label{prop:energy}
Let $(\vv_R,\QQ_R)$ be a smooth solution of the system (\ref{expan-BE-vR})--(\ref{expan-BE-QR}) on $[0,T]$.
 Then  for any $t\in [0,T]$,  it holds that
\begin{align}
\frac{d }{d t}\mathfrak{E}(t)+\Ff(t)\le
 C\big(1+\Ef+\ve^2\Ef+\ve^{14}\Ef^5\big)
 +C\big(\ve+\ve^2\Ef^{\f12}+\ve^4\Ef\big)\Ff.\non
 \end{align}
 \end{proposition}

To prove the proposition, we need the following lemmas.

\begin{lemma}\label{lem:energy}
It holds that
\begin{align}
&\|\QQ_R\|_{H^1}+\big\|(\ve\nabla^2\QQ_R, \ve^2\nabla^3\QQ_R)\big\|_{L^2}
+\big\|(\vv_R, \ve\nabla\vv_R, \ve^2\nabla^2\vv_R)\big\|_{L^2}\le C\Ef(t)^{\f12},\non\\
&\big\|(\ve^{-1}\CH_\nn^\ve(\QQ_R), \nabla\CH_\nn^\ve(\QQ_R), \ve\Delta\CH_\nn^\ve(\QQ_R))\big\|_{L^2}
+\big\|(\nabla\vv_R, \ve\nabla^2\vv_R, \ve^2\nabla^3\vv_R )\big\|_{L^2}\le C\big(\Ff(t)+\Ef(t)\big)^\f12.\non
\end{align}
\end{lemma}

\begin{proof}
The first inequality follows from Lemma \ref{lem:L}. By the commutator estimate
\begin{align}\label{commu}
\big\|[\nabla, \CH_\nn^\ve]\QQ_R\big\|_{L^2}\le C\|\QQ_R\|_{L^2},\quad
\big\|[\Delta, \CH_\nn^\ve]\QQ_R\big\|_{L^2}\le C\|\QQ_R\|_{H^1},
\end{align}
we have
\begin{align*}
&\|\nabla\CH_\nn^\ve(\QQ_R)\|_{L^2}\le \|\CH_\nn^\ve(\nabla\QQ_R)\|_{L^2}+C\|\QQ_R\|_{L^2},\\
&\|\ve\Delta\CH_\nn^\ve(\QQ_R)\|_{L^2}\le \|\ve\CH_\nn^\ve(\Delta\QQ_R)\|_{L^2}+C\ve\|\QQ_R\|_{H^1}.
\end{align*}
This gives the second inequality.
\end{proof}

%By Sobolev inequality and Lemma \ref{lem:energy}, we get
%\begin{align}
%&\|\ve\QQ_R\|_{L^\infty}+\|\ve^2\nabla\QQ_R\|_{L^\infty}+\|\ve^2\vv_R\|_{L^\infty} \le C\Ef(t)^{\f12},\label{eq:Qv-infty}\\
%& \|\ve\CH_\nn^\ve(\QQ_R)\|_{L^\infty}+\|\ve^2\nabla\vv_R\|_{L^\infty}\le C\big(\Ff(t)^\f12+\ve\Ef(t)^\f12\big).
%\end{align}
The following inequality will be useful for the estimates of $(\FF_R,\GG_R)$:
\ben\label{eq:product}
\|fg\|_{H^k}\le C\|f\|_{H^2}\|g\|_{H^k}\quad \textrm{for}\quad  k=0,1,2.
\een

\begin{lemma}\label{lem:FR}
It holds that
\begin{align}
\|(\FF_R,\ve\nabla\FF_R,\ve^2\Delta\FF_R)\|_{L^2}\le C\big(1+\Ef^{\f12}+\ve\Ff^\f12+
\ve\Ef+\ve^2\Ef^{\f12}\Ff^{\f12}+\ve^3\Ef^{\f32}+\ve^4\Ef\Ff^{\f12}\big).\non
\end{align}
\end{lemma}

\begin{proof}
By Lemma \ref{lem:energy}, it is easy to see that
\begin{align*}
\|(\FF_1,\ve\nabla\FF_1,\ve^2\Delta\FF_1)\|_{L^2}& \le C,\\
\|(\FF_2,\ve\nabla\FF_2,\ve^2\Delta\FF_2)\|_{L^2}& \le C\Ef^{\f12},\\
\|(\FF_3,\ve\nabla\FF_3,\ve^2\Delta\FF_3)\|_{L^2}& \le C\big(\Ef^\f12+\ve\Ff^\f12\big),
\end{align*}
and by (\ref{eq:product}), we get
\begin{align*}
\|(\FF_4,\ve\nabla\FF_4,\ve^2\Delta\FF_4)\|_{L^2}& \le C\ve\big(\Ef+\ve\Ef^{\f12}\Ff^{\f12}+\ve^2\Ef^{\f32}+\ve^3\Ef\Ff^{\f12}\big),\\
\|(\FF_5,\ve\nabla\FF_5,\ve^2\Delta\FF_5)\|_{L^2}& \le C\ve\big(\Ef+\ve^2\Ef^{\f32}\big).
\end{align*}
The lemma follows.
\end{proof}

\begin{lemma}\label{lem:GR}
It holds that
\begin{align*}
&\|(\GG_R, \ve\nabla\GG_R, \ve^2\Delta\GG_R)\|_{L^2}\le
C\big(1+\Ef^{\f12}+\ve\Ef+\ve^7\Ef^{\f52}+\ve\Ff^{\f12}
+\ve^2\Ef^{\f12}\Ff^{\f12}+\ve^4\Ef\Ff^{\f12}\big),\\
&\|(\GG'_R, \ve\nabla\GG'_R, \ve^2\Delta\GG'_R)\|_{L^2}\le C\big(1+\Ef^{\f12}+\Ff^{\f12}+\ve\Ef^{\f12}\Ff^{\f12}\big).
\end{align*}
\end{lemma}

\begin{proof}
The second inequality follows easily from (\ref{eq:product}) and Lemma \ref{lem:energy}. Obviously,
\begin{align*}
\|(\GG_1,\ve\nabla\GG_1,\ve^2\Delta\GG_1)\|_{L^2}\le C.
\end{align*}
And by (\ref{eq:product}) and Lemma \ref{lem:energy}, we have
\begin{align}
\|(\CJ_R^\ve,\ve\nabla\CJ_R^\ve,\ve^2\Delta\CJ_R^\ve)\|_{L^2}\le C\big(1+\Ef^{\f12}+
\ve\Ef+\ve^3\Ef^{\f32}\big),\non
\end{align}
which along with Lemma \ref{lem:energy} gives
\begin{align*}
&\|(\GG_2,\ve\nabla\GG_2,\ve^2\Delta\GG_2)\|_{L^2}\le C\big(\ve\Ff^{\f12}+\Ef^{\f12}\big),\\
&\|(\GG_3,\ve\nabla\GG_3, \ve^2\Delta\GG_3)\|_{L^2}\le C\big(1+\Ef^{\f12}+\ve\Ef+\ve^3\Ef^{\f32}\big).
\end{align*}
By (\ref{eq:product}), we get
\begin{align*}
\|\GG_4\|_{H^k}\le & C\ve^2 \|\ve\QQ_R\|_{H^2}\big(\|\ve^{-1}\HH_R\|_{H^k}+\|\CJ_R^\ve\|_{H^k}+\|\QQ_R\|_{H^k}\big)
+C\ve\|\ve^2\nabla\QQ_R\|_{H^2}\|\nabla\QQ_R\|_{H^k},
\end{align*}
which implies that
\begin{align*}
&\|(\GG_4,\ve\nabla\GG_4,\ve^2\Delta\GG_4)\|_{L^2}\\
&\le C\big(\Ef^{\f12}+\ve\Ef+\ve^3\Ef^{\f32}+\ve^5\Ef^2
+\ve^7\Ef^{\f52}+\ve\Ef^{\f12}\Ff^{\f12}+\ve^4\Ef\Ff^{\f12}\big).
\end{align*}
Summing up, we conclude the first inequality.
\end{proof}
\vspace{0.1cm}

Now we are in position to prove Proposition \ref{prop:energy}. The proof is split into four steps.\vspace{0.2cm}

{\bf Step 1. $L^2$ estimate}\vspace{0.1cm}

By (\ref{expan-BE-vR})--(\ref{expan-BE-QR}) and Lemma \ref{lem:energy}, we get
\begin{align}\nonumber
\Big\langle\frac{\partial \QQ_R}{\partial t}, \QQ_R\Big\rangle+\frac{1}{\Gamma\ve}\langle\CH_{\nn}^\ve(\QQ_R),\QQ_R\rangle&=\big\langle S_{\QQ_0}\DD_R+\BOm_R\cdot\QQ_0-\QQ_0\cdot\BOm_R+\FF_R,\QQ_R\big\rangle\non\\\nonumber
&\le C\|\QQ_R\|_{L^2}\big(\|\nabla\vv_R\|_{L^2}+\|\FF_R\|_{L^2}\big)\\
&\le C\Ef^{\f12}\big(\Ff^{\f12}+\|\FF_R\|_{L^2}\big),\label{eq:QR-L2-est}
\end{align}
and
\begin{align}\nonumber
\!\!\!\!\!\!&\Big\langle\frac{\partial \vv_R}{\partial t}, \vv_R\Big\rangle+\Big\langle
\frac{\partial \QQ_R}{\partial t}, \frac1{\ve}\CH_\nn^\ve(\QQ_R)\Big\rangle\\\nonumber
\!\!\!\!&=-\eta\langle\nabla\vv_R,\nabla\vv_R\rangle
-\Big\langle\frac1\ve S_{\QQ_0}(\CH_\nn^\ve(\QQ_R))
-\frac1\ve\QQ_0\cdot \CH_\nn^\ve(\QQ_R)+\frac1\ve\CH_\nn^\ve(\QQ_R)\cdot\QQ_0+\GG_R,~ \nabla\vv_R\Big\rangle\\\nonumber
&\quad+\langle\GG_R',\vv_R\rangle+\Big\langle-\frac{1}{\Gamma\ve}\CH_\nn^\ve(\QQ_R)
+S_{\QQ_0}\DD_R+\BOm_R\cdot\QQ_0-\QQ_0\cdot\BOm_R+\FF_R, ~\frac1\ve\CH_\nn^\ve(\QQ_R)\Big\rangle\\\nonumber
\!\!\!\!&=-\eta\langle\nabla\vv_R,\nabla\vv_R\rangle-\frac1\Gamma\Big\langle\frac1\ve\CH_\nn^\ve(\QQ_R), \frac1\ve\CH_\nn^\ve(\QQ_R)\Big\rangle
-\big\langle\GG_R,\nabla\vv_R\big\rangle+\langle\GG_R',\vv_R\rangle+\Big\langle\FF_R,\frac1\ve\CH_\nn^\ve(\QQ_R)\Big\rangle.
\end{align}
Here we used the following \textbf{important cancelation relation}
\begin{align}\nonumber
-\Big\langle\frac1\ve &S_{\QQ_0}(\CH_\nn^\ve(\QQ_R))-\frac1\ve\QQ_0\cdot
\CH_\nn^\ve(\QQ_R)+\frac1\ve\CH_\nn^\ve(\QQ_R)\cdot\QQ_0,~ \nabla\vv_R\Big\rangle\\
&+\Big\langle S_{\QQ_0}\DD_R+\BOm_R\cdot\QQ_0-\QQ_0\cdot\BOm_R+\FF_R,
~\frac1\ve\CH_\nn^\ve(\QQ_R)\Big\rangle=0. \label{cancel}
\end{align}
Thus, we obtain
\begin{align}\nonumber
&\Big\langle\frac{\partial \vv_R}{\partial t}, \vv_R\Big\rangle+\Big\langle
\frac{\partial \QQ_R}{\partial t}, \frac1{\ve}\CH_\nn^\ve(\QQ_R)\Big\rangle
+\eta\langle\nabla\vv_R,\nabla\vv_R\rangle+\frac1\Gamma\Big\langle\frac1\ve\CH_\nn^\ve(\QQ_R), \frac1\ve\CH_\nn^\ve(\QQ_R)\Big\rangle\\
&\le C\big(\|\GG_R'\|_{L^2}\Ef^{\f12}+(\|\GG_R\|_{L^2}+\|\FF_R\|_{L^2})\Ff^{\f12}\big).\label{eq:vR-L2-est}
\end{align}

{\bf Step 2. $H^1$ estimate}\vspace{0.1cm}

Using (\ref{expan-BE-vR})--(\ref{expan-BE-QR}) again, we get
\begin{align}\nonumber
&\ve^2\Big\langle\frac{\partial }{\partial t}\partial_i\vv_R, \partial_i\vv_R\Big\rangle
+\ve\Big\langle\frac{\partial}{\partial t}\partial_i\QQ_R, \CH_\nn^\ve(\partial_i\QQ_R)\Big\rangle+\ve^2\eta\big\langle\nabla\partial_i\vv_R,\nabla\partial_i\vv_R\big\rangle\\\nonumber
&=-\Big\langle\partial_i\big[S_{\QQ_0}(\CH_\nn^\ve(\QQ_R))
-\QQ_0\cdot \CH_\nn^\ve(\QQ_R)+\CH_\nn^\ve(\QQ_R)\cdot\QQ_0
+\ve\GG_R\big],~\ve\nabla\partial_i\vv_R\Big\rangle-\ve^2\big\langle\partial_i\widetilde{\vv}\cdot\nabla\vv_R, \partial_i\vv_R \big\rangle\\\nonumber
&\quad+\big\langle\ve\partial_i\GG_R',\ve\partial_i\vv_R\big\rangle+\ve\Big\langle\partial_i\big[
-\frac{1}{\Gamma\ve}\HH_R+S_{\QQ_0}\DD_R+\BOm_R\cdot\QQ_0
-\QQ_0\cdot\BOm_R+\FF_R\big], \CH_\nn^\ve(\partial_i\QQ_R)\Big\rangle.
\end{align}
The terms on the right hand sides are estimated as follows
\begin{align*}
&\big\langle\partial_i\big[S_{\QQ_0}(\CH_\nn^\ve(\QQ_R))
-\QQ_0\cdot \CH_\nn^\ve(\QQ_R)+\CH_\nn^\ve(\QQ_R)\cdot\QQ_0\big], ~\ve\nabla\partial_i\vv_R\big\rangle\\
&\le \big\langle S_{\QQ_0}(\partial_i\CH_\nn^\ve(\QQ_R))
-\QQ_0\cdot \partial_i\CH_\nn^\ve(\QQ_R)+\partial_i\CH_\nn^\ve(\QQ_R)\cdot\QQ_0, ~\ve\nabla
\partial_i\vv_R\big\rangle\\
&\qquad+C\|\CH_\nn^\ve(\QQ_R)\|_{L^2}\|\ve\Delta\vv_R\|_{L^2}\\
&\le \big\langle S_{\QQ_0}(\CH_\nn^\ve(\partial_i\QQ_R))
-\QQ_0\cdot \CH_\nn^\ve(\partial_i\QQ_R)+\CH_\nn^\ve(\partial_i\QQ_R)\cdot\QQ_0, ~
\ve\nabla\partial_i\vv_R\big\rangle\\
&\qquad+C\big(\|\QQ_R\|_{L^2}+\|\CH_\nn^\ve(\QQ_R)\|_{L^2}\big)\|\ve\Delta\vv_R\|_{L^2},\\
&\ve\Big\langle-\frac{1}{\Gamma\ve}\partial_i\CH_\nn^\ve(\QQ_R),
\CH_\nn^\ve(\partial_i\QQ_R)\Big\rangle\le -\frac1\Gamma\|\CH_\nn^\ve(\partial_i\QQ_R)\|_{L^2}^2+C\|\QQ_R\|_{L^2}\|\CH_\nn^\ve(\partial_i\QQ_R)\|_{L^2},\\
&\ve\big\langle\partial_i\big[S_{\QQ_0}\DD_R+\BOm_R\cdot\QQ_0-\QQ_0\cdot\BOm_R\big],
\CH_\nn^\ve(\partial_i\QQ_R)\big\rangle\\
&\le \ve\big\langle S_{\QQ_0}\partial_i\DD_R+\partial_i\BOm_R\cdot\QQ_0-\QQ_0\cdot\partial_i\BOm_R,
\CH_\nn^\ve(\partial_i\QQ_R)\big\rangle + C\|\ve\nabla\vv_R\|_{L^2}\|\CH_\nn^\ve(\partial_i\QQ_R)\|_{L^2},
\end{align*}
and
\begin{align*}
&\ve^2\big\langle\partial_i\GG_R, \nabla\partial_i\vv_R\big\rangle\le C\|\ve\partial_i\GG_R\|_{L^2}\|\ve\nabla\partial_i\vv_R\|_{L^2},\\
&\ve^2\big\langle\partial_i\GG'_R, \partial_i\vv_R\big\rangle\le C\|\ve\partial_i\GG_R'\|_{L^2}\|\ve\partial_i\vv_R\|_{L^2},\\
&\ve\big\langle\partial_i\FF_R, \CH_\nn^\ve(\partial_i\QQ_R)\big\rangle\le C\|\ve\partial_i\FF_R\|_{L^2}\| \CH_\nn^\ve(\partial_i\QQ_R)\|_{L^2}.
\end{align*}
Thus by (\ref{cancel}) and Lemma \ref{lem:energy}, we get
\begin{align}\nonumber
&\ve^2\Big\langle\frac{\partial }{\partial t}\partial_i\vv_R, \partial_i\vv_R\Big\rangle
+\ve\Big\langle\frac{\partial}{\partial t}\partial_i\QQ_R, \CH_\nn^\ve(\partial_i\QQ_R)\Big\rangle\\\nonumber
&\le-\ve^2\eta\langle\nabla\partial_i\vv_R,\nabla\partial_i\vv_R\rangle-\frac{1}{\Gamma}\big\langle\CH_\nn^\ve(\partial_i\QQ_R),\CH_\nn^\ve(\partial_i\QQ_R)\big\rangle
\\&\quad+C\big(\Ef+\Ef^{\f12}\Ff^{\f12}+\ve\Ff\big)+C\|\ve\partial_i\GG_R'\|_{L^2}\Ef^{\f12}
+C\big(\|\ve\partial_i\GG_R\|_{L^2}+\|\ve\partial_i\FF_R\|_{L^2}\big)\Ff^{\f12}.\label{eq:vR-H1-est}
\end{align}

{\bf Step 3. $H^2$ estimate}\vspace{0.1cm}

Since the proof is very similar to Step 2, we omit the details. We have
\begin{align}\nonumber
&\ve^4\Big\langle\frac{\partial }{\partial t}\Delta\vv_R, \Delta\vv_R\Big\rangle+\ve^3\Big\langle
\frac{\partial}{\partial t}\Delta\QQ_R, \CH_\nn^\ve(\Delta\QQ_R)\Big\rangle\\\nonumber
&\le-\ve^4\eta\big\langle\nabla\Delta\vv_R,\nabla\Delta\vv_R\big\rangle
-\frac{\ve^2}{\Gamma}\big\langle\CH_\nn^\ve(\Delta\QQ_R),\CH_\nn^\ve(\Delta\QQ_R)\big\rangle
\\&\quad+C\big(\Ef+\Ef^{\f12}\Ff^{\f12}+\ve\Ff\big)+C\|\ve^2\Delta\GG_R'\|_{L^2}\Ef^{\f12}
+C\big(\|\ve^2\Delta\GG_R\|_{L^2}+\|\ve^2\Delta\FF_R\|_{L^2}\big)\Ff^{\f12}.\label{eq:vR-H2-est}
\end{align}

{\bf Step 4. The completion of energy estimate}\vspace{0.1cm}

Due to $\QQ_R:\II=\tr\QQ_R=0$, we have
\begin{eqnarray}\nonumber
\frac1\ve\frac{d }{dt}\big\langle\QQ_R, \CH_\nn^\ve(\QQ_R)\big\rangle
&=&\frac2\ve\big\langle\frac{\partial}{\partial t}\QQ_R,~ \CH_\nn^\ve(\QQ_R)\big\rangle
+\frac1\ve\Big\langle\QQ_R,~bs\big(\partial_t(\nn\nn)\cdot\QQ_R+\QQ_R\cdot\partial_t(\nn\nn)\big)\\\nonumber
&&\qquad-2cs^2\big[\QQ_R:\partial_t(\nn\nn)\big](\nn\nn)-2cs^2(\QQ_R:\nn\nn)\partial_t(\nn\nn)\Big\rangle\\
&=&\frac2\ve\big\langle\frac{\partial}{\partial t}\QQ_R,~ \CH_\nn^\ve(\QQ_R)\big\rangle
+\frac2\ve\big\langle\QQ_R,~bs\partial_t(\nn\nn)\cdot\QQ_R-2cs^2\QQ_R:\partial_t(\nn\nn)(\nn\nn)\big\rangle.\nonumber
\end{eqnarray}
We infer from Lemma  \ref{lem:control}  that
\begin{align}\nonumber
&\frac2\ve\big\langle\QQ_R,~bs\partial_t(\nn\nn)\cdot\QQ_R-2cs^2\QQ_R:\partial_t(\nn\nn)(\nn\nn)\big\rangle\\
&\qquad\le \delta\|\frac1\ve\CH_\nn^\ve\QQ_R\|_{L^2}^2+C\Big(\frac1\ve\langle\CH_\nn^\ve\QQ_R,\QQ_R\rangle+\|\QQ_R\|_{L^2}^2\Big).\non
\end{align}
Therefore, we have
\begin{align*}
\frac1{2\ve}\frac{d }{dt}\big\langle\QQ_R, \CH_\nn^\ve(\QQ_R)\big\rangle
\le \frac1\ve\big\langle\frac{\partial}{\partial t}\QQ_R, \CH_\nn^\ve(\QQ_R)\big\rangle
+\delta\Ff+C\Ef.
\end{align*}
Similarly, we can obtain
\begin{align}
\f\ve 2\frac{d }{dt}\big\langle\partial_i\QQ_R, \CH_\nn^\ve(\partial_i\QQ_R)\big\rangle
\le& \ve\big\langle\frac{\partial}{\partial t}\partial_i\QQ_R, \CH_\nn^\ve(\partial_i\QQ_R)\big\rangle
+\delta\Ff+C\Ef,\nonumber\\
\f{\ve^3}2\frac{d }{dt}\big\langle\Delta\QQ_R, \CH_\nn^\ve(\Delta\QQ_R)\big\rangle
\le& \ve^3\big\langle\frac{\partial}{\partial t}\Delta\QQ_R, \CH_\nn^\ve(\Delta\QQ_R)\big\rangle
+\delta\Ff+C\Ef.\nonumber
\end{align}

Summing up (\ref{eq:QR-L2-est}) and (\ref{eq:vR-L2-est})--(\ref{eq:vR-H2-est}),  we infer from Lemma \ref{lem:FR} and Lemma \ref{lem:GR}  that
\begin{align}
\f12\f d{dt}\Ef(t)+\Ff(t)\le  C(1+\Ef+\ve^2\Ef+\ve^{14}\Ef^5)+(\delta+C\ve+C\ve^2\Ef^{\f12}+C\ve^4\Ef)\Ff.\non
\end{align}
Then the proposition follows by taking $\delta$ small.

\section{Proof of Theorem \ref{thm:main}}

Given the initial data $(\vv_0^\ve,\QQ_0^\ve)\in H^2\times H^3$, it can be showed by the energy method \cite{PZ2} that
there exists $T_\ve>0$ and a unique solution $(\vv^\ve,\QQ^\ve)$ of the system (\ref{eq:BE-ve})--(\ref{eq:BE-Qe}) such that
\beno
\vv^\ve\in C([0,T_\ve];H^2)\cap L^2(0,T_\ve;H^3),\quad \QQ^\ve\in C([0,T_\ve];H^3)\cap L^2(0,T_\ve;H^4).
\eeno
Thanks to $\HH(\QQ), S_{\QQ}(\DD)\in \mathbb{Q}$  for $\QQ\in \mathbb{Q}$, we have $\QQ^\ve\in \mathbb{Q}$.
Moreover, by Proposition \ref{prop:Hilbert}, the solution has the expansion
\begin{align*}
&\vv^\ve=\vv_0+\ve\vv_1+\ve^2\vv_2+\ve^3\vv_R^\ve,\\
&\QQ^\ve=\QQ_0+\ve\QQ_1+\ve^2\QQ_2+\ve^3\QQ_3+\ve^3\QQ_R^\ve.
\end{align*}
For the remainder $(\vv_R^\ve,\QQ_R^\ve)$, we infer from Proposition \ref{prop:energy} that
\begin{align}
\frac{d }{d t}\mathfrak{E}(t)+\Ff(t)\le
 C\big(1+\Ef+\ve^2\Ef+\ve^{14}\Ef^5\big)
 +C\big(\ve+\ve^2\Ef^{\f12}+\ve^4\Ef\big)\Ff,\non
 \end{align}
for any $t\in [0,T_\ve]$. Thanks to the assumptions of Theorem \ref{thm:main}, we know that
$\Ef(0)\le CE_0$. Thus, there exists $\ve_0, E_1>0$ depending on $T, \vv, \nn,  E_0$ such that
for any $\ve\in (0,\ve_0)$ and $t\in [0,\min(T,T_\ve)]$,
\beno
\Ef(t)+\int_0^t\Ff(s)d s\le E_1.
\eeno
This in turn implies that $T_\ve\ge T$ by a continuous argument.
Then Theorem \ref{thm:main} follows.

\section*{Acknowledgments}
P. Zhang is partly supported by NSF of China under Grant 50930003 and 21274005.
Z. Zhang is partially supported by NSF of China under Grant 10990013 and 11071007,
Program for New Century Excellent Talents in University and Fok Ying Tung Education Foundation.

\end{document}